\documentclass[reqno,a4paper,11pt]{amsart}

\usepackage[T1]{fontenc}
\usepackage[utf8]{inputenc}
\usepackage{lmodern}
\usepackage[english]{babel}

\usepackage[%
    a4paper,
	left=2.5cm,       
	right=2.5cm,      
	top=3.5cm,        
	bottom=3.5cm,     
	heightrounded,    
	bindingoffset=0mm 
]{geometry}

\usepackage{amsmath}
\usepackage{amssymb}
\usepackage{amsfonts}
\usepackage{mathtools}
\usepackage{mathrsfs}
\usepackage{enumerate}
\usepackage{xcolor}

\usepackage[]{hyperref}
\hypersetup{
    colorlinks=true,       
    linkcolor=blue,          
    citecolor=magenta,        
    filecolor=magenta,      
    urlcolor=cyan           
}

\usepackage[noabbrev,capitalize]{cleveref}

\usepackage{bookmark}
\usepackage[initials]{amsrefs}

\theoremstyle{plain}
\newtheorem{theorem}{Theorem}[section]
\newtheorem{lemma}[theorem]{Lemma}
\newtheorem{prop}[theorem]{Proposition}

\newtheorem{corollary}[theorem]{Corollary}

\theoremstyle{definition}

\theoremstyle{remark}
\newtheorem{remark}[theorem]{Remark}

\numberwithin{equation}{section}

\newcommand{\eps}{\varepsilon}

\newcommand{\R}{\mathbb R}
\newcommand{\N}{\mathbb N}
\newcommand{\T}{\mathbb T}

\newcommand{\pot}{\phi}

\DeclareMathOperator{\diver}{div}

\title[On Double H\"older regularity of pressure in bounded domains]{On Double H\"older regularity of the hydrodynamic pressure in bounded domains
}

\author[L.~De Rosa]{Luigi De Rosa}
\address[L.~De Rosa]{Department Mathematik und Informatik, Universität Basel, Spiegelgasse~1, 4051 Basel, Switzerland}
\email{luigi.derosa@unibas.ch}

\author[M.~Latocca]{Micka\"el Latocca}
\address[M.~Latocca]{Department Mathematik und Informatik, Universität Basel, Spiegelgasse~1, 4051 Basel, Switzerland}
\email{mickael.latocca@unibas.ch}

\author[G.~Stefani]{Giorgio Stefani}
\address[G.~Stefani]{Scuola Internazionale Superiore di Studi Avanzati (SISSA), via Bonomea~265, 34136 Trieste (TS), Italy}
\email{giorgio.stefani.math@gmail.com}


\date{\today}


\thanks{\textit{Acknowledgments}. 
The first two authors are partially supported by the 2015 ERC Grant 676675 FLIRT--Fluid Flows and Irregular Transport. 
The second author thanks Nicolas Burq, Paul Dario and Jules Pertinand for several interesting discussions he had with them.
The third author is a member of Istituto Nazionale di Alta Matematica (INdAM), Gruppo Nazionale per l'Analisi Matematica, la Probabilità e le loro Applicazioni (GNAMPA) and has received funding from the European Research Council (ERC) under the European Union’s Horizon 2020 research and innovation program (grant agreement No.~945655).
}

\thanks{\textit{Data availability statement}.
Data sharing not applicable to this article as no datasets were generated or analysed during the current study.}

\allowdisplaybreaks
\begin{document}

\begin{abstract}
We prove that the hydrodynamic pressure $p$ associated to the velocity $u\in C^\theta(\Omega)$, $\theta\in(0,1)$, of an inviscid incompressible fluid in a bounded and simply connected domain $\Omega\subset \R^d$ with $C^{2+}$ boundary satisfies $p\in C^{\theta}(\Omega)$ for $\theta \leq \frac12$ and $p\in C^{1,2\theta-1}(\Omega)$ for $\theta>\frac12$. Moreover, when $\partial \Omega\in C^{3+}$,  we prove that an almost double H\"older regularity $p\in C^{2\theta-}(\Omega)$ holds even for $\theta<\frac12$. This extends and improves the recent result of~\cite{BT21} obtained in the planar case to every dimension $d\ge2$ and it also doubles the pressure regularity.
Differently from~\cite{BT21}, we do not introduce a new boundary condition for the pressure, but instead work with the natural one.
In the boundary-free case of the $d$-dimensional torus, we show that the double regularity of the pressure can be actually achieved under the weaker assumption that the divergence of the velocity is sufficiently regular, thus not necessarily zero.
\end{abstract}

\maketitle

\section{Introduction}
\label{sec:intro}
Let $d\geq 2$ and let $\Omega\subset\R^d$ be a bounded and simply connected domain of class $C^2$.  The time evolution in $\Omega$ of an incompressible inviscid fluid is described by the Euler equations
\begin{equation}\label{E}
\left\{\begin{array}{rcll}
\partial_t u+ \diver (u \otimes u) +\nabla p &=& 0 &\text{in } \Omega\times (0,T)\\[1mm]
\diver u& = &0&\text{in } \Omega\times (0,T)\\[1mm]
u\cdot n&=&0&\text{on } \partial \Omega\times (0,T)
\end{array}\right.
\end{equation}
where $u\colon\Omega\times(0,T)\rightarrow\R^d$ and $p\colon\Omega\times(0,T)\rightarrow\R$ are the \emph{velocity} of the fluid and its \emph{hydrodynamic pressure} respectively and $n\colon\partial\Omega\to\R^d$ the outward unit normal to $\partial \Omega$. The boundary condition $u(\cdot, t)\cdot n=0$ on $\partial \Omega$ is the usual \emph{no-flow condition}, which prohibits the fluid to escape from the spatial domain $\Omega$, being it always tangential to the boundary.

\subsection{The pressure equation}
In this article, we focus on the pressure $p$.
Taking the divergence of the first equation in~\eqref{E}, we get
\begin{equation}
\label{eq:squirrel}
-\Delta p(\cdot,t)=\diver \diver (u(\cdot,t)\otimes u(\cdot, t)) \quad \text{in } \Omega
\end{equation}
for all $t\in (0,T)$.
Being interested in the spatial regularity of $p$, from now on we will fix a time slice $t$ and consider the vector field $u(\cdot,t)$. Thus, for simplicity, we will drop the explicit time dependence and we will just write $u=u(x)$. 
In particular, our results can be applied for every fixed $t$-time slice. 

In a bounded domain $\Omega$, we clearly need to complement the (interior) elliptic equation~\eqref{eq:squirrel} with an appropriate boundary condition, which in the case of tangential boundary condition takes the form 
\begin{equation}
\label{eq:nut}
\partial_n p\,=u\otimes u :\nabla n \quad
\text{on } \partial\Omega,
\end{equation}
where we implicitly assumed the normal $n$ to be extended to a $C^1$ vector field in a neighborhood of $\partial \Omega$ in order to compute its gradient.  Note that this makes sense if the domain is of class $C^2$. 

The Neumann-type boundary condition~\eqref{eq:nut} can be obtained if we scalar multiply by~$n$ the first equation in~\eqref{E}. Indeed, since $u$ is divergence-free and tangential to the boundary, at least in the case in which the velocity is regular enough, we can compute
\begin{equation}
\label{reassembl_bound_cond}
\begin{split}
\partial_n p&=\nabla p\cdot n
=-\diver (u\otimes u)\cdot n =-\partial_i(u_i u_j)n_j
=-u_i\partial_i(u_j ) n_j
\\
&=-u_i\partial_i(u_j n_j)+u_iu_j\partial_i n_j
=-u\cdot\nabla(u\cdot n)+u\otimes u:\nabla n
\\
&=u\otimes u:\nabla n. 
\end{split}
\end{equation}
Here and in the rest of the paper,   we adopt the Einstein summation convention with repeated indexes.
In the last equality in the above chain, we used that $\partial \Omega$ is a level set of the scalar function $u\cdot n$, and thus $\nabla (u\cdot n)|_{\partial \Omega}$ is parallel to $n$.  
Thus, at least in the regular setting, the pressure $p$ solves
\begin{equation}\label{p_problem}
\left\{\begin{array}{rcll}
-\Delta p &=& \diver \diver (u\otimes u) & \text{in } \Omega\\[2mm]
\partial_n p\, &=&u\otimes u : \nabla n & \text{on } \partial\Omega.
\end{array}\right.
\end{equation}

In order to deal with $u\in C^\theta(\Omega)$,  we need to interpret \eqref{p_problem} in the weak sense, that is, we consider a scalar function $p\in C^0(\overline \Omega)$ such that 
\begin{equation}\label{p_weaksol}
-	\int_{\Omega} p\, \Delta \varphi \,dx+\int_{\partial \Omega} p \,\partial_n \varphi\,dx = \int_{\Omega} u\otimes u : H \varphi\,dx,
\qquad \text{for all } \varphi \in C^2(\overline \Omega),
\end{equation}
where we denoted by $H \varphi$ the Hessian matrix of the scalar function $\varphi$. Relation \eqref{p_weaksol} is obtained as usual by multiplying the first equation in \eqref{p_problem} by the test function $\varphi$ and then integrating by parts.  

At this point one may wonder if the pressure $p$ coming from the Euler equations solves~\eqref{p_problem} only if the velocity is regular enough, say $u(t)\in C^1(\overline \Omega)$. Indeed the latter regularity has been crucially used in order to derive the boundary condition in \eqref{reassembl_bound_cond}. However, it can be easily shown that the weak formulation of the pressure equation \eqref{p_weaksol} can be always derived whenever $u,p\in C^0(\overline \Omega)$.
Indeed, if the  uniformly continuous  couple $(u,p)$ weakly solves~\eqref{E}, then 
$$
\int_0^T\left( \int_{\Omega} \left(u\cdot \psi  \,\eta' +\eta\, u\otimes u:\nabla \psi+\eta \,p\diver \psi \right) dx-\eta \int_{\partial \Omega} p\psi\cdot n \,dx \right)dt=0
$$
whenever $\eta\in C^\infty_c((0,T))$ and $\psi\in C^1(\overline \Omega;\R^d)$. 
In particular, choosing $\psi=\nabla \varphi$, we get that 
$$
\int_0^T\eta \left( \int_{\Omega} \left( u\otimes u: H\varphi+p\Delta \varphi \right) dx- \int_{\partial \Omega} p\partial_n\varphi \,dx \right)dt=0
$$
for all $\varphi\in C^2(\overline \Omega)$, so that, for any fixed time slice $t\in (0,T)$, the couple $(u(t),p(t))$ must solve~\eqref{p_weaksol}.

\subsection{Regularity of the pressure}

From the heuristic standard elliptic regularity theory applied to the Neumann problem \eqref{p_problem}, one is tempted to say that the pressure $p$ has (at least) the same H\"older regularity of $u$ up to the boundary.
However, as noticed in~\cite{BT21}, this regularity property does not directly come from any known elliptic regularity result. Moreover, it is also known~\cites{SILV,CD18,Is2013} that, if $\Omega$ is either the whole space $\R^d$ or the $d$-dimensional torus $\T^d$ (i.e., no boundary is involved), then the pressure enjoys the following regularity properties
\begin{equation}\label{p_double_regular}
p\in \left\{\begin{array}{ll}
C^{2\theta} &\text{ if } 0<\theta<\frac{1}{2}\\[2mm]
C^{1,2\theta-1} & \text{ if } \frac{1}{2}<\theta<1.
\end{array}\right.
\end{equation}
Moreover, as observed in~\cite{C2014}, in the critical case $\theta=\frac{1}{2}$, the pressure $p$ is log-Lipschitz, namely
$$
|p(x_1)-p(x_2)| \leq C|x_1-x_2|\,|\log|x_1-x_2||
$$
for all $x_1,x_2\in\Omega$, with $|x_1-x_2|<\frac12$.
Similar results can be also shown in the classes of Besov and Sobolev  solutions, see~\cite{CDF20}.  

The regularity in~\eqref{p_double_regular} basically means that $p$ is twice as regular as $u$. 
Thus, even in a bounded domain $\Omega$, by a standard localization argument, one immediately gets that the pressure still enjoys an \emph{interior} double regularity property.  

On the other hand, in the recent work~\cite{BT21}, the authors proved that, in the $2$-dimensional case, the pressure satisfies $p\in C^\theta(\Omega)$, thus providing the $\theta$-H\"older regularity of the pressure \emph{up to the boundary}.

\subsection{Main results}

The aim of this work is to extend the boundary regularity of the pressure to every dimension $d\geq 2$ and to show that  the $2\theta$-H\"older regularity known for domains without boundary, still holds in setting considered in this work.
Note that, when $\theta >\frac{1}{2}$, in order to hope for a regularity property $p\in C^{1,2\theta-1}(\Omega)$, we are forced to require that $\partial \Omega$ is of class $C^{2,2\theta-1}$.
In particular, this implies that $\nabla n\in C^{2\theta-1}$, so that also the Neumann boundary data satisfies 
$$u\otimes u:\nabla n\in C^{\min\{\theta,2\theta-1\}}=C^{2\theta-1}.$$ 
In a bounded domain with a $C^2$ boundary this would clearly not be possible in general, since in this case the  Neumann boundary condition $u\otimes u:\nabla n\in C^{0}$ is inconsistent with the fact that $p$ has H\"older continuous first derivatives up to the boundary. 

Our main result reads as follows.

\begin{theorem}[Regularity of the pressure]
\label{t_main}
Let $d\geq 2$ and let  $\Omega\subset \R^d$ be a bounded simply connected open set with boundary of class $C^{2,\alpha}$, for some $\alpha>0$.   
Let $\theta\in (0,1)$ and let $u\in C^\theta(\Omega)$ be a weakly divergence-free vector field such that $u\cdot n|_{\partial \Omega}=0$.  
Then, there exists a unique zero-average solution $p\in C^0(\overline \Omega)$ of~\eqref{p_weaksol} with the following regularity properties.
\begin{enumerate}[(i)]
\item
\label{item:main<12}
If $\theta\in\left(0,\frac{1}{2}\right]$, then $p\in C^\theta(\Omega)$ and there exists a constant $C>0$ such that 
$$
\|p\|_{C^{\theta}(\Omega)}\leq C \|u\|_{C^0(\Omega)} \|u\|_{C^\theta(\Omega)}.
$$

\item 
\label{item:main>12}
If $\theta\in\left(\frac{1}{2},1\right)$, then $p\in C^{1,\min(\alpha,2\theta-1)}(\Omega)$ and there exists a constant $C>0$ such that 
 $$
\|p\|_{C^{1,\min\{\alpha,2\theta-1\}}(\Omega)}\leq C \|u\|_{C^\theta(\Omega)}^2.
$$
In particular, if $\Omega$ is of class $C^{2,2\theta-1}$, then $p\in C^{1,2\theta-1}(\Omega)$.
\end{enumerate}
\end{theorem}

The stability estimate in~\eqref{item:main<12} clearly indicates that only one of the two vectors $u$ in the right-hand side is used to transfer the $C^\theta$ regularity to $p$.
In fact, we expect that the other vector could be exploited in a better way in order to double the regularity of the pressure.

The assumption on the $C^{2,\alpha}$ smoothness of $\partial \Omega$ for some (arbitrary small) $\alpha>0$ is just technical.
Indeed, this regularity is only needed to provide a suitable approximation for the vector field $u$ and to prove the asymptotic estimates for the Green--Neumann function, see  \cref{approx.lemma} and \cref{l:greenestimate} below. We believe that this assumption might not be sharp, since the $C^2$ smoothness of the boundary of the domain (at least heuristically) seems to be sufficient for proving~\eqref{item:main<12}. 

Actually, as noticed in~\cite{BT21} in the planar case $d=2$, the approximation of the velocity field can be performed in a $C^2$ domain by using the \emph{stream function} of~$u$, which plays the role of a potential.
In the case $d>2$, we do not know how to provide an appropriate corresponding $d$-dimensional version of the stream-function approach.
We refer the reader to \cref{remark_C2} below for a precise discussion about the issue of the boundary regularity of the domain together with a possible strategy to solve it.  

The approach used in~\cite{BT21} is based on a suitable geodesic parametrization of the boundary (together with the aforementioned regularization of~$u$) which, in turn, strongly relies on the $2$-dimensional structure.

Our approach is different and follows the main idea of~\cite{CD18}.
Indeed, we rewrite $p$ as a singular integral operator applied to a suitable data of the form $(u-a)\otimes u$, where $a$ is a $d$-dimensional vector needed to desingularize the kernel.  
More precisely, we can state the following result.

\begin{prop}[Representation formula]\label{p_representation}
Let $d\geq 2$ and let $\Omega\subset \R^d$ be a bounded simply connected  open set of class $C^2$. 
Let $u\in C^\infty(\Omega)\cap C^1(\overline \Omega)$ be such that $\diver u=0$  and $u\cdot n|_{\partial \Omega}=0$. 
If $p$ is a weak solution of \eqref{p_problem}, then 
\begin{align}\label{formula_for_p}
p(x)-\frac{1}{|\Omega|}\int_{\Omega}p(y)\,dy&=\int_{\Omega} \partial_{y_iy_j}G(x,y)\,(u_i(y)-u_i(x))\,u_j(y)\,dy
\end{align}
for all $x\in\Omega$, where $G=G(x,y)$ is the Green--Neumann function on $\Omega$.
\end{prop}

Note that the right-hand side of~\eqref{formula_for_p} makes sense since the singularity of the kernel $y\mapsto\partial_{y_iy_j} G(x,y)$ at $y=x$ can be resolved by the term $u_i(y)-u_i(x)$. 

We also remark that looking at weak solutions in the sense of~\eqref{p_weaksol} plays a crucial role in our approach, since in this way there is no need to introduce any other boundary condition apart from the natural one as in~\eqref{p_problem}.
This is in fact different from what happens in~\cite{BT21}, where the authors need to define and deal with a specific notion of trace of the normal derivative of the pressure $\partial_n p$ at the boundary of $\Omega$.

We apply \cref{p_representation} to the regular approximation of $u\in C^\theta(\Omega)$ given by the approximation  \cref{approx.lemma}.
We then show that the corresponding H\"older norm of $p$ is uniformly bounded by the norm of $u$, thus getting the regularity estimates claimed in \cref{t_main} for the solution $p$ of~\eqref{p_weaksol}.\\

Moreover, despite the fact that we are not able to quadratically desingularize the kernel by using the representation formula~\eqref{formula_for_p} above, we propose an abstract interpolation argument which shows that an \emph{almost} double regularity holds even for $\theta<\frac12$.  
However, this abstract approach requires more regularity on the domain $\Omega$.

\begin{theorem}[Almost double regularity for $\partial \Omega\in C^{3,\alpha}$]\label{t:p_almost_double}
Let $d\geq 2$ and let $\Omega\subset \R^d$ be a bounded simply connected open set with boundary of class $C^{3,\alpha}$, for some $\alpha>0$. Let $\eps>0$, $\theta\in \left(0,\frac12 \right)$  and $u\in C^\theta(\Omega)$ be a weakly divergence-free vector field such that $u\cdot n |_{\partial \Omega}=0$. Then, there exists a constant $C>0$ such that the unique zero-average solution $p$ of~\eqref{p_weaksol} enjoys the estimate
\begin{equation}\label{est_p_almost_double}
\|p\|_{C^{2\theta-\eps}(\Omega)}\leq C \|u\|^2_{C^\theta(\Omega)}.
\end{equation}
\end{theorem}

The proof of the previous theorem is based on the abstract interpolation result given in~\cite{CDF20}*{Theorem~3.5} together with the $\theta$-H\"older regularity we obtained in \cref{t_main}~\eqref{item:main<12}.  
The assumption on the $C^{3,\alpha}$ regularity of the domain is needed in order to apply the abstract interpolation~\cite{CDF20}*{Theorem 3.5}, because we must guarantee that $p\in C^{2,\alpha}(\Omega)$ whenever $u\in C^{1,\alpha}(\Omega)$.
Note that this condition is indeed true by classical Schauder's estimates applied to the problem~\eqref{p_problem} whenever the boundary datum satisfies $u\otimes u\colon\nabla n\in C^{1,\alpha}(\partial \Omega)$, which, in turn, requires that $\nabla n\in C^{1,\alpha}(\partial \Omega)$ and thus $\partial \Omega\in  C^{3,\alpha}$.  
The general heuristic reason behind the validity of the estimate~\eqref{est_p_almost_double} lies in the fact that, once the \emph{single} $\theta$-regularity estimate given by \cref{t_main}\eqref{item:main<12} is established (and we underline that this regularity is not a direct consequence of the classical Schauder's estimates), the structure of the right-hand side
$$
\diver \diver (u\otimes u)=\diver (u\cdot \nabla u)=(\partial_i u_j) (\partial_j u_i)
$$
always upgrades the regularity to a (almost) double one, if the domain is regular enough to properly apply Schauder's theory. 
The $\eps$-loss in~\eqref{est_p_almost_double} is a consequence of the failure of Schauder's estimates in the integer spaces $C^k(\Omega)$.
\\

Last but not least, we show that the quadratic desingularisation of the kernel allowing for the double regularity result~\eqref{p_double_regular} on the $d$-dimensional torus $\T^d$ holds even if $\diver u \neq 0$, provided that $\diver u$ is suitably regular.
In more precise terms, we can prove the following result.

\begin{theorem}[Double regularity on $\mathbb T^d$ for general $\diver u$]
\label{double_torus}
Let $\theta\in (0,1)$ and $u\in C^\theta(\T^d)$.
The unique zero-average solution $p\in C^\theta(\T^d)$ of the problem
\begin{equation}
\label{p_torus}
-\Delta p=\diver \diver (u\otimes u)\quad \text{in } \T^d
\end{equation}
enjoys the following regularity properties.
\begin{enumerate}[(a)]

\item\label{item:torus<1/2}
If $\theta\in\left(0,\frac{1}{2}\right)$ and $\diver u\in L^q(\T^d)$ for some $q\in\left[\frac{2d}{1-2\theta},+\infty\right]$, then $p\in C^{2\theta}(\T^d)$ and there exists a constant $C>0$ such that 
$$
\|p\|_{C^{2\theta}(\T^d)}\leq C \left(\|u\|^2_{C^\theta(\T^d)}+ \| \diver u\|^2_{L^q (\T^d)} \right).
$$

\item\label{item:torus>1/2}
If $\theta\in\left(\frac{1}{2},1\right)$ and $\diver u\in C^{2\theta-1}(\T^d)$, then $p\in C^{1,2\theta-1}(\T^d)$ and there exists a constant $C>0$ such that 
$$
\|p\|_{C^{1,2\theta-1}(\T^d)}\leq C\left( \|u\|^2_{C^\theta(\T^d)}+ \| \diver u\|^2_{C^{2\theta-1}(\T^d)}\right).
$$
\end{enumerate}
\end{theorem}

\subsection{H\"older solutions and turbulence}

H\"older-continuous weak solutions to the Euler equations have attracted a lot of interest in the last decades, mainly because of their natural connection with the  \emph{K41 Theory of Turbulence}~\cite{K41} and the related Onsager's conjecture on \emph{anomalous energy dissipation}. 
Indeed, in 1949, the theoretical physicist Onsager~\cite{Ons49}, while considering solutions $u\in L^\infty ([0,T];C^\theta(\T^d))$, conjectured that $\theta=\frac13$ may be the threshold for the existence of inviscid fluid flows  exhibiting an anomalous dissipation of the kinetic energy. 
He in fact claimed---and actually even mathematically motivated---that kinetic energy dissipation would only happen in the range $\theta<\frac13$, while for $\theta>\frac13$ the existence of such badly behaved solutions would not be possible due to some intrinsic rigidity properties of the equations.  We refer the interested reader to \cites{CET94,Ey94,CCFS08,BDLSV2019,Is2018,DH21,DT19} and references therein for a complete overview about the proof of Onsager's conjecture, and to \cite{F} for a detailed account on the modern theory of fully developed turbulence.

Everything that has been discussed above is known to be true in the absence of boundaries. 
If a physical boundary is present, then the mathematical (and also physical) description of turbulence becomes much more intricate, due to the non-trivial effects made by the boundary itself \cites{BT18,BTW19,RRS18,RRS2018}.  The energy conservation analogous to~\cite{CET94} in the case of a bounded domain was proved in~\cite{BT18}, where the H\"older regularity of the pressure plays a crucial role, see~\cite{BT18}*{Proposition~2.1} for instance.  
However, such a regularity does not easily come as a consequence of the usual Schauder boundary regularity theory for elliptic PDEs. For this reason, we believe that this article gives a solid proof of the pressure regularity which, to our knowledge, is not otherwise precisely traceable in the current mathematical literature.  

\subsection{Organization of the paper}

The paper is organized as follows.
In \cref{s:approximation}, we prove an approximation result for the velocity, see \cref{approx.lemma}. 
\cref{s:representation} is dedicated to the proof of the representation formula given in \cref{p_representation}.
The proof of our main result \cref{t_main}  can be found in \cref{s:t_main_proof}. 
The proof of the  almost double regularity of \cref{t:p_almost_double} is the content of \cref{sec:almostdouble} while \cref{double_torus} for non-zero divergence is proved in \cref{non-zero_div}.
We then end the paper with three appendices. 
In \cref{s:schauder}, we collect all the technical results we need on Schauder regularity theory. In \cref{s:bilin_interp} we recall an useful result about interpolation of bilinear operators.
In \cref{greenfunctionestimates}, we provide a proof of some estimates on the Green--Neumann function.

\section{Approximation of the velocity and reduction to the regular case}
\label{s:approximation}

\subsection{H\"older norms}
We begin with the definition of the norms we are going to use throughout the paper. 
In the following, we let $\N_0=\N\cup\{0\}$, $k\in\N_0$, $\theta\in (0,1)$ and $\beta\in\N_0^k$ be a multi-index.  
Let $\Omega\subset \R^d$ and let $f\colon\Omega\rightarrow \R^m$ for $m\in\N$. We introduce the usual H\"older norms as follows. 
The supremum norm is denoted by $\|f\|_{C^0(\Omega)}=\sup_{x\in \Omega}|f(x)|$. We define the H\"older seminorms 
as
\begin{equation*}
\begin{split}
[f]_{C^k(\Omega)}&=\max_{|\beta|=k}\|D^{\beta}f\|_{C^0(\Omega)}\, ,\\
[f]_{C^{k,\theta}(\Omega)} &= \max_{|\beta|=k}\sup_{x,y\in \Omega,  \,x\neq y}\frac{|D^{\beta}f(x)-D^{\beta}f(y)|}{|x-y|^{\theta}}\, .
\end{split}
\end{equation*}
The H\"older norms are then given by
\begin{eqnarray*}
\|f\|_{C^k(\Omega)}&=&\sum_{j=0}^k[f]_{C^j(\Omega)},\\
\|f\|_{C^{k,\theta}(\Omega)}&=&\|f\|_{C^k(\Omega)}+[f]_{C^{k,\theta}(\Omega)}.
\end{eqnarray*}
In order to shorten the notation, for $k=0$ we simply write $C^\theta(\Omega)$ instead of $C^{0,\theta}(\Omega)$. 

\subsection{Approximation of the velocity}
Now we provide the regular approximation of the velocity $u$ that remains divergence-free and tangential to the boundary. These two constraints are both important in order to get the representation formula \eqref{formula_for_p}.

\begin{lemma}[Approximation of the velocity]\label{approx.lemma} 
Let $d\ge2$ and let $\Omega\subset\R^d$ be a bounded and simply connected domain of class $C^{2,\alpha}$ for some $\alpha>0$. Let $\theta \in (0,1)$ and let $u\in C^{\theta}(\Omega)$ be such that $\diver  u=0$ and $u\cdot n\vert_{\partial\Omega}=0$.  Then, there exists a family $(u^{\varepsilon})_{\eps>0} \subset C^{\infty}(\Omega)\cap C^{1}(\overline\Omega)$ such that
$u^{\varepsilon}\to u\ \text{in}\ C^0(\overline \Omega)$
as $\eps\to0^+$,
$\diver   u^{\varepsilon}=0$ and $u^{\varepsilon}\cdot n \vert_{\partial \Omega}=0
$ for all $\eps>0$, and
\begin{equation*}
\sup_{\eps>0}\|u^{\varepsilon}\|_{C^{\theta}(\Omega)} \leq C \|u\|_{C^{\theta}(\Omega)}
\end{equation*} 
for some constant $C>0$.
\end{lemma}

In the proof of \cref{approx.lemma}, we exploit the following extension result for H\"older continuous solenoidal vector fields, see~\cite{KMPT00}*{Section 5}.

\begin{lemma}[Extension lemma]\label{extension.lemma} Let $\Omega$ be a bounded and simply connected domain of class~$C^2$ and let $\theta \in (0,1)$.
If $u \in C^{\theta}(\Omega)$ is such that $\diver  u=0$, then there exists $\tilde{u} \in C^{\theta}(\mathbb{R}^d)$ with compact support such that $\tilde{u}\vert_{\Omega}=u$,  $\diver \tilde u=0$ in $\R^d$, and 
\begin{equation*}
\|\tilde{u}\|_{C^{\theta}(\mathbb{R}^d)} \leq C\|u\|_{C^{\theta}(\Omega)}
\end{equation*}
for some constant $C>0$.  
\end{lemma}

\begin{proof}[Proof of \cref{approx.lemma}] Let $\tilde{u}$ be the extension given by  \cref{extension.lemma} and consider the mollification $\tilde{u}_{\varepsilon} :=\tilde{u}*\rho_{\varepsilon}$ for some approximation of the unity $(\rho_{\varepsilon})_{\varepsilon >0}$. Notice that $\tilde u^\eps$ does not satisfy the required boundary condition in general. Hence, we modify it by considering the solution $\varphi^{\varepsilon}$ (unique up to constants) of the system
\[
    \left\{
    \begin{array}{lcll}
         \Delta \varphi^{\varepsilon} &=&0 &\text{in }\Omega\\
         \partial_{n} \varphi^{\varepsilon}&=& \tilde{u}_\varepsilon \cdot n &\text{on }\partial \Omega.
    \end{array}
    \right. 
\]
Notice that, since $\diver \tilde u_\eps=0$, the \emph{compatibility condition}
$$
0=\int_{\partial \Omega} \tilde u_\eps \cdot n=\int_{\Omega} \diver \tilde u_\eps
$$
is satisfied, and thus the previous boundary value problem admits a solution. 

Since $\Omega$ is of class $C^{2,\alpha}$,  we have $n\in C^{1,\alpha}$ and thus, by Schauder boundary regularity theory (see~\cite{N14}) we get $\varphi^\varepsilon \in C^\infty(\Omega)\cap C^{2,\alpha}(\Omega)\subset C^\infty(\Omega)\cap C^2(\overline \Omega)$. 
Moreover, by  \cref{nardi-1}, we also get
\begin{equation}\label{est_phi_1theta}
    \|\nabla \varphi^{\varepsilon}\|_{C^{\theta}(\Omega)}  \leq C \|\tilde u_\varepsilon\cdot n\|_{C^\theta(\partial\Omega)}\leq C \|\tilde{u}_{\varepsilon}\|_{C^{\theta}(\mathbb{R}^d)}\leq C \|\tilde{u}\|_{C^{\theta}(\mathbb{R}^d)}\leq C \|u\|_{C^\theta(\Omega)},
    \end{equation}
where in the last inequality we used the continuity of the extension operator given by \cref{extension.lemma}.

We now define $u^{\varepsilon} := \tilde{u}_{\varepsilon} - \nabla \varphi^{\varepsilon}$ for all $\eps>0$.
Note that, by definition, $u^\eps\in C^{\infty}(\Omega)\cap C^{1}(\overline \Omega)$, $\diver   u^{\varepsilon} = 0$ in $\Omega$ and $u^{\varepsilon}\cdot n = 0$ on $\partial\Omega$ for all $\eps>0$. It remains to check the convergence $u^{\varepsilon} \to u$ in $C^{0}(\overline\Omega)$ as $\eps\to0^+$. 

We first note that $\tilde{u}_{\varepsilon} \to \tilde{u}$ in $C^0(\mathbb{R}^d)$ as $\eps\to0^+$, which implies $\tilde{u}_{\varepsilon} \to u$ in $C^{0}(\overline \Omega)$ as $\eps\to0^+$, because $\tilde{u}|_{\overline \Omega}=u$. We are thus to show that $\nabla \varphi^{\varepsilon} \to 0$ in $C^{0}(\overline \Omega)$ as $\eps\to0^+$. By~\eqref{est_phi_1theta}, we know that the family $(\varphi^{\varepsilon})_{\varepsilon >0}$ is bounded in $C^{1,\theta}(\Omega)$.
Therefore, we can find a subsequence $\varepsilon_k \to 0$ and a limit function $\varphi \in C^{1,\theta}(\Omega)$ such that $\varphi ^{\varepsilon _k} \to \varphi $ in $C^1(\overline\Omega)$ as $k\to+\infty$. 
Finally, since $\tilde{u}_{\varepsilon}\cdot n \to 0$ in $C^0(\partial \Omega)$ as $\eps\to0^+$, the limit function $\varphi$ solves 
\[
    \left\{
    \begin{array}{rcll}
         \Delta \varphi &=& 0 &\text{in }\Omega\\
         \partial_{n} \varphi &=&0 &\text{on }\partial \Omega.
    \end{array}
    \right. 
\]
As a consequence, $\varphi$ must be a constant function on $\Omega$ and thus $\nabla \varphi^{\varepsilon_k} \to\nabla \varphi \equiv 0$ uniformly on $\overline \Omega$ as $k\to+\infty$.  From~\eqref{est_phi_1theta}, we also get the (uniform in $\eps$) continuity estimate
\[
    \|u^\varepsilon\|_{C^\theta(\Omega)}\leq \|\tilde u_\varepsilon\|_{C^\theta(\Omega)}+ \|\nabla \varphi^{\varepsilon}\|_{C^{\theta}(\Omega)}\leq C\|u\|_{ C^\theta(\Omega)},
\]
concluding the proof. 
\end{proof}

\begin{remark}\label{r:C1alpha_approx}
Notice that, in the previous proof, $\varphi^\eps\in C^{2,\alpha}(\Omega)$, which implies the stronger conclusion $u^\eps\in C^{1,\alpha}(\Omega)\cap C^\infty(\Omega)$.  It follows that, actually, $\diver \diver(u^\eps\otimes u^\eps)=\partial_i u^\eps_j\partial_j u^\eps_i\in C^\alpha (\Omega)$, and thus the Neumann boundary problem for the pressure precisely falls in the standard Schauder's theory with H\"older continuous data.  We however preferred to formulate \cref{approx.lemma} by claiming that $u^\eps\in C^1(\overline \Omega)$ only, since this is the minimal assumption needed to run the rest of the proof. 
\end{remark}

\subsection{A possible strategy to get approximation in the \texorpdfstring{$C^2$}{Cˆ2} case}
\label{remark_C2}

In the proof of \cref{approx.lemma}, the regularity $\partial \Omega \in C^{2,\alpha}$ is only used to deduce that $\varphi^\varepsilon\in C^2(\overline \Omega)$, which in turn implies that $u^\varepsilon\in C^1(\overline \Omega)$. 

It is clear that, in order to get an approximation $(u^\eps)_{\eps}\subset C^1(\overline \Omega)$ with the properties given in \cref{approx.lemma}, one is forced to ensure that $\nabla \varphi^\eps\in C^1(\overline\Omega)$ which, in turn, is a consequence of standard Schauder's boundary regularity estimate provided that $\Omega$ is of class $C^{2,\alpha}$ for some $\alpha>0$, see~\cite{GT}*{Theorem~6.30} for instance.  

Thus, to prove the same approximation under the weaker assumption that $\Omega$ is of class~$C^2$ only, one has to follow a different approach. 
Below we propose a possible  strategy which we believe could be helpful for further developments in this direction.

The idea is to write $u\in C^{\theta}(\Omega)$ in terms of a potential with some precise properties. 
Let us assume the existence of a $d \times d$ matrix $A$ such that  
\begin{equation}\label{propertiesA}
A\in C^{1,\theta}(\Omega),  \quad
A^{T}=-A, 
\quad 
A|_{\partial \Omega}=0,
\quad
\diver A=u.
\end{equation}
Note that the latter condition is clearly compatible with the divergence-free constraint on~$u$, since $\diver \diver A=0$ whenever $A$ is skew-symmetric.  

In order to regularize $A$, one may proceed as follows.
In the interior of $\Omega$, one can simply mollify the potential $A$.
At the boundary $\Omega$, one first reduces to the case $\Omega$ is the half-space $\{x\in\R^d : x_d>0\}$ through local charts, and then extend the local expression of $A$ for $x_d<0$ in an odd-symmetric way.
Let us now call such a local extension $\tilde A$. 
At least in a local chart, one can mollify $\tilde A$ with a radial smooth convolution kernel, getting a new potential $\tilde A_\eps$. Since $\tilde A$ was odd with respect to $x_d=0$ (in each local system of coordinates), from $A|_{\partial \Omega}=0$ we get $\tilde A_\eps|_{x_d =0}=0$.
All in all, via a suitable smooth partition of the unity subordinated to the local boundary charts and the interior, one can go back to the original coordinates in $\Omega$ and just glue all those local potentials together, getting a new $d\times d$ matrix $A^\eps$ in (an open neighborhood of) $\overline\Omega$ such that 
\begin{equation}
\label{propertiesA-eps}
A^\eps\in C^2(\overline\Omega),
\quad
(A^\eps)^T=-A^\eps,
\quad
A^\eps|_{\partial\Omega}=0,
\end{equation}  
the first property being a consequence of the fact that the domain~$\Omega$ is of class~$C^2$.  

At this point, one can define 
$$u^\eps=\diver A^\eps\in C^1(\overline \Omega)$$
and we claim that the family $(u^\eps)_{\eps>0}$ provides the desired approximation. 
By construction, we clearly have that $u^\eps\to u$ in $C^0(\overline \Omega)$ as $\eps\to0^+$ and $\|u^{\varepsilon}\|_{C^{\theta}(\Omega)} \leq C \|u\|_{C^{\theta}(\Omega)}$ for some constant $C>0$ which does not depend on $\eps>0$.
In addition, from~\eqref{propertiesA-eps}, one immediately gets that $\diver u^\eps=0$.
Finally, one also gets $u^\eps\cdot n=0$ on $\partial\Omega$
thanks to the following straightforward computation
\[
    n\cdot \diver A^\eps=n_i\, \partial_j A^\eps_{ji}=n_i\alpha^\eps_{ji}n_j=-n_i\alpha^\eps_{ij}n_j=-n\cdot \diver A^\eps \quad \text{on } \partial \Omega,
\]
where we used that $A^\eps$ is skew-symmetric and that $\partial \Omega$ is a level set for $A^\eps$, so that $\partial_j A^\eps_{ji}=\alpha^\eps_{ji} n_j$ for some suitable constants $\alpha^\eps_{ji}$, being $\nabla A^\eps_{ij}|_{\partial \Omega}$ parallel to $n$ for all~$i,j$.

At the present moment, we do not know how to provide a potential $A$ as in~\eqref{propertiesA} in a general $d$-dimensional domain $\Omega$ of class $C^2$. 
Having in mind the closely related works~\cites{BP19,BP20}, the existence of such a potential seems quite delicate and this is why we instead decided to rely on the the simpler argument leading to  \cref{approx.lemma}.  
Nonetheless, in the planar case~$d=2$, one can overcome the problem and provide the suitable approximation~$(u^\eps)_{\eps>0}$ by simply relying on the well-known \emph{stream function} $\Psi\in C^{1,\theta}(\Omega)$ vanishing on~$\partial \Omega$, see~\cite{BT21}*{Lemma~1} for the details.

\section{Proof of \texorpdfstring{\cref{p_representation}}{the representation formula}}

\label{s:representation}

In this section, we prove \cref{p_representation}, that is, we establish the representation formula~\eqref{formula_for_p}. 
To this aim, let $u\in C^\infty(\Omega)\cap C^1(\overline \Omega)$ be such that $\diver u=0$ and $u\cdot n|_{\partial\Omega}=0$ and let $p$ be a weak solution of~\eqref{p_problem}.  
We start by rewriting the right-hand side of \eqref{p_problem} as 
\[
    \diver \diver (u\otimes u)=\partial_{ij}(u_iu_j)=\partial_ju_i\,\partial_iu_j=\partial_j(u_i-u_i(x))\,\partial_iu_j=\partial_{ij}\big((u_i-u_i(x))\,u_j \big),
\]
where $x\in\Omega$ is any given point.
It is well known (see~\cite{S16}*{Theorem~3.39} for instance) that $p$ can be represented via the formula
\begin{equation}
\label{ping}
p(x)-\frac{1}{|\Omega|}\int_{\Omega}p(y)\,dy=\int_{\Omega} G(x,y) \,\partial_{ij}\big( (u_i-u_i(x))\,u_j\big)\,dy+\int_{\partial \Omega} G(x,y)\, u_i\,u_j\,\partial_j n_i\,d\sigma(y),
\end{equation}
where $G=G(x,y)$ is the \emph{Green--Neumann function} on $\Omega$ as defined in \cref{greenfunctionestimates} and $d\sigma$ is the surface measure on $\partial \Omega$.

Integrating by parts the first term in the above representation, we get
\begin{equation}
\label{pong}
\begin{split}
\int_{\Omega} G(x,y)\, &\partial_{ij}\big( (u_i-u_i(x))\,u_j(y)\big)\,dy
=
\int_{\partial \Omega} G(x,y)\,\partial_{j}\big( (u_i-u_i(x))\,u_j\big)\,n_i\,d\sigma(y) \\
&\quad
-\int_{\Omega}\partial_{y_i}G(x,y) \,\partial_j\big( (u_i-u_i(x))\,u_j\big)\,dy\\
&=\int_{\partial \Omega} G(x,y)\,u_j\,\partial_{j}u_i\,n_i\,d\sigma(y)-\int_{\partial \Omega} \partial_{y_i} G(x,y)\,\big( (u_i-u_i(x))\,u_j\big)\,n_j\,dy\\
&\quad+\int_{\Omega} \partial_{y_iy_j}G(x,y) \,(u_i-u_i(x))\,u_j\,dy\\
&=-\int_{\partial \Omega}  G(x,y)\, u\otimes u : \nabla n\,d\sigma(y)+\int_{\Omega} \partial_{y_iy_j}G(x,y)\, (u_i-u_i(x))\,u_j\,dy,
\end{split}
\end{equation}
where in the last equality we used that $u\cdot n=0$ and $\left((u\cdot \nabla)u \right)\cdot n=-u\otimes u :\nabla n$ on~$\partial \Omega$,  as already done to recover the boundary condition in~\eqref{reassembl_bound_cond}. 
Thus, by combining~\eqref{ping} with~\eqref{pong}, we obtain
\begin{align*}
p(x)-\frac{1}{|\Omega|}\int_{\Omega}p(y)\,dy&=\int_{\Omega} \partial_{y_iy_j}G(x,y)(u_i-u_i(x))u_j\,dy
\end{align*}
proving the desired representation formula~\eqref{formula_for_p}.

\section{Proof of \texorpdfstring{\cref{t_main}}{the main result}}
\label{s:t_main_proof}

In this section, we prove \cref{t_main}, dealing with the two cases $\theta\leq \frac12$ and $\theta>\frac12$ separately.
We remark that, since in \cref{t_main} the pressure is assumed to have zero average, we just need to prove a bound on the seminorm $[p]_{C^\theta(\Omega)}$.
This is in fact a simple consequence of the Poincare-type inequality
$$
\|f\|_{C^0(\Omega)}\leq C[f]_{C^\theta(\Omega)}
$$
valid for all $f\colon\Omega\to\R$ such that $\int_{\Omega}f(x)\,dx=0$, where $C>0$ is a constant depending on~$\theta$ and~$\Omega$ only.
  
We start by noticing that, for every fixed $\eps>0$, by the classical theory there exists  a unique zero average solution $p^\eps\in C^{1,\alpha}(\Omega)$ to \eqref{p_problem}, with right-hand side given by the approximation $u^\varepsilon$ of \cref{approx.lemma} in place of $u$.

In the case $\theta\leq \frac12$, as a consequence of the estimates on the Green--Neumann function established in \cref{greenfunctionestimates}, we show that  such $p^\varepsilon$ satisfies the estimate in~\eqref{item:main<12} of \cref{t_main} uniformly with respect to $\varepsilon>0$.  Therefore, since $C^{\theta}(\Omega)$ is compactly embedded in $C^0(\overline \Omega)$, up to subsequences as $\varepsilon\to0^+$, we get the existence of limit function~$p$ still satisfying the estimate    in \cref{t_main}~\eqref{item:main<12} which is also a weak solution of~\eqref{p_problem} in the sense of~\eqref{p_weaksol}.

In the case $\theta>\frac12$, we notice that the normal derivative of the pressure on the boundary is a well-defined $C^{2\theta-1}$ function.
Thus, we can just extend~$u$ to the whole space $\R^d$ and then exploit the double regularity proved in  \cite{CD18}*{Proposition 3.1} together with some standard Schauder regularity estimates.

\subsection{Case \texorpdfstring{$\theta \leq \frac12$}{below 1/2}}

In order to keep the notation short, we write $u$ in place of the approximation $u^\eps$ given by \cref{approx.lemma}, that is, we have $u\in C^{1}(\overline \Omega) \cap C^{\infty}(\Omega)$, $\diver u=0$ in $\Omega$ and $u\cdot n =0$ on $\partial\Omega$. 

Since we assumed that the pressure is average-free, by \cref{p_representation}, we can write 
\[
    p(x)=\int_{\Omega} \partial_{y_iy_j}G(x,y)\, (u_i(y)-u_i(x))\,u_j(y)\,dy
\]
for any given $x\in\Omega$.
Now let $x_1, x_2 \in \Omega$, $\bar x := \frac{x_1+x_2}{2}$ and $\lambda := |x_1-x_2|$.
Since by~\cite{CD18} we already know that $p\in C^\theta$ in the interior of $\Omega$, we can just focus on the case in which one between the two points, say $x_1$ is close to $\partial \Omega$. Without loss of generality, up to possibly enlarge the constants appearing in the final estimate, we can assume that $\lambda \leq \lambda_0$ for some $\lambda_0>0$ so small that such that $B(\bar x,\lambda_0) \cap \partial\Omega$ is the graph of a $C^{2,\alpha}$ function. 
In addition, possibly choosing $\lambda_0>0$ smaller, we can assume that, for all $\lambda\le\lambda_0$, the intersection $B(\bar x,\lambda)\cap \Omega$ is a simply connected open set with Lipschitz boundary (so that, in the following computations, the Divergence Theorem always applies).
We start by writing 
\begin{align*}
    p(x_1)-p(x_2)&=\int_{\Omega} \partial_{y_iy_j}G(x_1,y) \,(u_i(y)-u_i(x_1))\,u_j(y)\,dy \\
    & \quad- \int_{\Omega} \partial_{y_iy_j}G(x_2,y)\, (u_i(y)-u_i(x_2))\,u_j(y)\,dy = A + B,
\end{align*}
where 
\begin{align*}
    A &:= \int_{\Omega\cap B(\bar x, \lambda)} \partial_{y_iy_j}G(x_1,y)\, (u_i(y)-u_i(x_1))\,u_j(y)\,dy \\ 
    &\quad- \int_{\Omega\cap B(\bar x, \lambda)} \partial_{y_iy_j}G(x_2,y)\, (u_i(y)-u_i(x_2))\,u_j(y)\,dy
\end{align*}
and
\begin{align*}
    B &:= \int_{\Omega\setminus B(\bar x, \lambda)} \partial_{y_iy_j}G(x_1,y) \,(u_i(y)-u_i(x_1))\,u_j(y)\,dy \\
    &\quad- \int_{\Omega\setminus B(\bar x, \lambda)} \partial_{y_iy_j}G(x_2,y)\, (u_i(y)-u_i(x_2))\,u_j(y)\,dy.
\end{align*}

We start by estimating the  terms in $A$.
We provide the detailed computations for the case $d\geqslant 3$, then case  $d=2$ being similar with minor differences due to the different expression of the Newtonian potential and, consequently, of the Green--Neumann function.

By using the Green--Neumann function estimates in~\eqref{eq.pointwiseOrder12}, we can bound 
\begin{align*}
    \bigg|\int_{\Omega\cap B(\bar x, \lambda)} \partial_{y_iy_j}G(x_k,y)\, &(u_i(y)-u_i(x_k))\,u_j(y)\,dy \,\bigg|
    \\
&\lesssim \|u\|_{C^{\theta(\Omega)}}\|u\|_{C^0(\Omega)} \int_{\Omega\cap B(\bar x, \lambda)} \frac{dy}{|x_k-y|^{d-\theta}}
\end{align*}
for $k=1,2$.
Since
\begin{equation*}
    \int_{\Omega\cap B(\bar x, \lambda)} \frac{dy}{|x_k-y|^{d-\theta}}\leq \int_{ B(x_k, 2\lambda)} \frac{dy}{|x_k-y|^{d-\theta}} \lesssim   \lambda^{\theta},
\end{equation*}
we get 
\begin{equation}
    \label{eq.termI}
    |A| \lesssim \lambda ^{\theta}\, \|u\|_{C^{\theta}(\Omega)}\|u\|_{C^0(\Omega)}. 
\end{equation}
To handle the terms in $B$, we write
\begin{align*}
    B &= \int_{\Omega\setminus B(\bar x, \lambda)} \big(\partial_{y_iy_j}G(x_1,y)\,-\partial_{y_iy_j}G(x_2,y)\big)\,(u_i(y)-u_i(x_1))\,u_j(y) \,dy \\
    &\quad+\int_{\Omega\setminus B(\bar x, \lambda)}\partial_{y_iy_j}G(x_2,y)\,(u_i(x_2)-u_i(x_1))\,u_j(y)\,dy \\ 
    & =: B_1+B_2.  
\end{align*}
Since
$$\left|\partial_{y_iy_j}\big(G(x_1,y)-G(x_2,y)\big)\right| \lesssim \frac{|x_1-x_2|}{|\bar x - y|^{d+1}}$$ 
thanks to~\eqref{eq.pointwiseDifference}, we can estimate~$B_1$ as
\[
    |B_1| \lesssim \lambda \,\|u\|_{C^{\theta}(\Omega)}\|u\|_{C^0(\Omega)} \int_{\Omega \setminus B(\bar x, \lambda)} \frac{1}{|\bar x -y|^{d+1-\theta}}\,dy.
\]
Since
\begin{equation*}
    \int_{\Omega \setminus B(\bar x, \lambda)} \frac{dy}{|\bar x -y|^{d+1-\theta}}\,dy \lesssim \lambda^{\theta -1}, 
\end{equation*}
we conclude that
\begin{equation}
    \label{eq.termII}
    |B_1| \lesssim \lambda ^{\theta}\, \|u\|_{C^{\theta}(\Omega)}\|u\|_{C^0(\Omega)}.
\end{equation}

We are thus left to estimate $B_2$.
To this aim, we integrate by parts  another time to further desingularize the Green--Neumann kernel (otherwise, we would only obtain a \emph{logarithmic} $C^{\theta}$ regularity). 
By the Divergence Theorem, we can write
\begin{align*}
    B_2&= - \int_{\Omega\setminus B(\bar x,\lambda)} \partial_{y_i} G(x_2,y)\,(u_i(x_2)-u_i(x_1))\,\partial_ju_j(y)\,dy \\
    &+ \int_{\partial(\Omega \setminus B(\bar x, \lambda))} \partial_{y_i} G(x_2,y)\,(u_i(x_2)-u_i(x_1))\,u_j(y)\,n_j(y)\,d\sigma (y),
\end{align*}
where $\sigma$ stands for the $(d-1)$-dimensional Hausdorff measure on the boundary.
Now, since $u$ is divergence-free, the first integral in $B_2$ vanishes.  Moreover, we can decompose 
$$\partial(\Omega \setminus B(\bar x, \lambda)) = (\partial\Omega \setminus B(\bar x, \lambda)) \cup (\Omega \cap \partial B(\bar x, \lambda)),$$
where the two sets on the right-hand side are disjoint. 
The integral on $\partial\Omega \setminus B(\bar x, \lambda)$ vanishes because $u\cdot n =0$ on $\partial\Omega$.
We hence have to estimate 
\[
    B_2=\int_{\Omega \cap \partial B(\bar x, \lambda)} \partial_{y_i} G(x_2,y)\,(u_i(x_2)-u_i(x_1))\,u_j(y)\,n_j(y)\,d\sigma (y).
\]
By using \eqref{eq.pointwiseOrder12}, we get
\[
    |B_2|\lesssim \lambda ^{\theta}\, \|u\|_{C^{\theta}(\Omega)}\|u\|_{C^0(\Omega)}\int_{\Omega \cap \partial B(\bar x, \lambda)} \frac{d\sigma (y)}{|x_2-y|^{d-1}}.
\]
Since $|x_2-y| \gtrsim \lambda$ whenever $y \in \partial B(\bar x, \lambda)$, we can bound 
\[
    \int_{\Omega \cap \partial B(\bar x, \lambda)} \frac{d\sigma (y)}{|x_2-y|^{d-1}} \lesssim \frac{1}{\lambda^{d-1}}\,\sigma(\partial B(\bar x, \lambda))\lesssim 1,
\]
so that 
\begin{equation}
    \label{eq.termIII}
    |B_2| \lesssim \lambda ^{\theta}\, \|u\|_{C^{\theta}(\Omega)}\|u\|_{C^0(\Omega)}.
\end{equation}
Gathering~\eqref{eq.termI}, \eqref{eq.termII} and \eqref{eq.termIII} and recalling that $\lambda=|x_1-x_2|$, we finally obtain that
\[
    |p(x_1)-p(x_2)| \lesssim \|u\|_{C^{\theta}(\Omega)}\|u\|_{C^0(\Omega)} \,|x_1-x_2|^{\theta}
\]
as soon as $|x_1-x_2| \leq \lambda_0$.
The proof of \cref{t_main}~\eqref{item:main<12} is thus complete.

\subsection{Case \texorpdfstring{$\theta>\frac12$}{theta>1/2}}

Let $\tilde u\in C^\theta(\R^d)$ be the divergence-free and compactly supported extension of $u$ given by  \cref{extension.lemma}. We  let $\tilde p$ be the unique potential-theoretic solution of 
$$
-\Delta \tilde p= \diver \diver (\tilde u\otimes \tilde u) \quad \text{in } \R^d.
$$
Then, by \cite{CD18}*{Proposition~3.1}, we get 
\begin{equation}\label{est_ptilde}
\|\tilde p\|_{C^{1,2\theta-1}(\R^d)}\leq C \|\tilde u\|_{C^\theta(\R^d)}^2\leq C \| u\|_{C^\theta(\Omega)}^2.
\end{equation}
In particular, the normal derivative $\partial_n\tilde p$ is of class $C^{2\theta-1}$ on $\partial \Omega$. Now, the function 
$$q:=p -\tilde p + \int_{\Omega} \tilde p\,dx$$ 
satisfies the \emph{compatibility condition} $\int_{\partial \Omega}\partial_n q\,d\sigma=0$ by the definition and is thus the unique (zero-average) solution of
\[
    \left\{
    \begin{array}{rcll}
        -\Delta q&=&0&\text{in } \Omega\\
        \partial_n q&=& u\otimes u : \nabla n-\partial_n \tilde p & \text{on }\partial \Omega.
    \end{array}
    \right.
\]
Note that the Neumann boundary datum satisfies $$u\otimes u : \nabla n-\partial_n \tilde p\in C^{\min\{\alpha,2\theta-1\}}(\partial \Omega),$$ since $u\in C^\theta(\Omega)\subset C^{2\theta-1}(\Omega)$ and $\nabla n\in C^{\alpha}(\partial\Omega)$,  because $\Omega$ is of class $C^{2,\alpha}$.  
Thus, from \cref{nardi-1}, we get
\begin{equation}\label{est_q}
\|q\|_{C^{1,\min\{\alpha,2\theta-1\}}(\Omega)}\leq C\left( \|u \otimes u\|_{C^\theta(\partial \Omega)} + \|\partial_n\tilde p\|_{C^{2\theta-1}(\partial \Omega)}\right)\leq C \| u\|_{C^\theta(\Omega)}^2,
\end{equation}
where in the last inequality we also used~\eqref{est_ptilde}.  
Rewriting 
$$p=q+\tilde p-\int_{\Omega} \tilde p\,dx$$ and exploiting~\eqref{est_ptilde} and~\eqref{est_q}, we hence get that
$$
\|p\|_{C^{1,\min\{\alpha,2\theta-1\}}(\Omega)}\leq \|q\|_{C^{1,\min\{\alpha,2\theta-1\}}(\Omega)}+\|\tilde p\|_{C^{1,2\theta-1}(\Omega)}\leq C\|u\|_{C^\theta(\Omega)},
$$
concluding the proof of \cref{t_main}~\eqref{item:main>12}.

\section{Proof of \texorpdfstring{\cref{t:p_almost_double}}{Theorem 1.3}}\label{sec:almostdouble}

The goal is to apply \cref{t:bil_interp}. For notational convenience, we set 
$$
C_\sigma^\theta(\Omega)=\left\{ u\in C^\theta(\Omega) \, : \, \diver u=0 \text{ and } u\cdot n|_{\partial \Omega}=0 \right\}
$$
and similarly $C_\sigma^{1,\theta}(\Omega)$.
We define the bilinear operator $T=T(u,v)$ which maps every couple of vector fields $u,v\in C_\sigma^\theta(\Omega)$ to the unique zero-average solution of 
\begin{equation}\label{bilinear_operat}
\left\{\begin{array}{rcll}
-\Delta T(u,v) &=& \diver \diver (u\otimes v) & \text{in } \Omega\\[2mm]
\partial_n T(u,v)\, &=&u\otimes v : \nabla n & \text{on } \partial\Omega,
\end{array}\right.
\end{equation}
which, as for \eqref{p_weaksol},  in the weak formulation reads as 
\begin{equation}\label{weak_bil_operat}
-	\int_{\Omega} T(u,v)\, \Delta \varphi \,dx+\int_{\partial \Omega} T(u,v) \,\partial_n \varphi\,dx = \int_{\Omega} u\otimes v : H \varphi\,dx,
\qquad \text{for all } \varphi \in C^2(\overline \Omega).
\end{equation}
Let $\eps$, $\theta$ and $ \alpha$ be as in the statement of \cref{t:p_almost_double}.
Let $\beta>0$ be sufficiently small such that $\beta<\min\{\alpha,\theta,\eps \}$. By the very same proof of \cref{t_main}\eqref{item:main<12}, we get
\begin{equation}
\label{est_interp_double_1}
\|T(u,v)\|_{C^\beta(\Omega)}\leq C \|u\|_{C^0(\Omega)} \|v\|_{C^\beta(\Omega)}\leq C \|u\|_{C^\beta(\Omega)} \|v\|_{C^\beta(\Omega)}.
\end{equation}
Moreover, if $v\in C_\sigma^{1,\beta}(\Omega)$, we can rewrite
$$
\diver \diver (u\otimes v)=\diver (u\cdot \nabla v),
$$
which, together with \cref{nardi-1}, gives
\begin{equation}
\label{est_interp_double_2}
\|T(u,v)\|_{C^{1,\beta}(\Omega)}\leq C \|u\|_{C^\beta(\Omega)} \|v\|_{C^{1,\beta}(\Omega)}.
\end{equation}
Since $u\otimes v :H\varphi =v\otimes u :H\varphi $, by looking at the weak formulation \eqref{weak_bil_operat} we get that the bilinear operator $T=T(u,v)$ is symmetric, so that
\begin{equation}
\label{est_interp_double_3}
\|T(v,u)\|_{C^{1,\beta}(\Omega)}\leq C \|u\|_{C^\beta(\Omega)} \|v\|_{C^{1,\beta}(\Omega)}.
\end{equation}
Now, if both $u,v\in  C_\sigma^{1,\beta}(\Omega)$,  we can write
$$
\diver \diver (u\otimes v)=\partial_i u_j\partial_jv_i
$$
and thus, by classical Schauder's estimates, we can infer
\begin{equation}
\label{est_interp_double_4}
\|T(u,v)\|_{C^{2,\beta}(\Omega)}\leq C \|u\|_{C^{1,\beta}(\Omega)} \|v\|_{C^{1,\beta}(\Omega)}.
\end{equation}
Notice that, in order to obtain \eqref{est_interp_double_4}, we need to have $\beta<\alpha$, from which $\nabla n\in C^{1,\alpha}(\partial \Omega)\subset C^{1,\beta}(\partial \Omega)$ and thus the Neumann boundary datum in \eqref{bilinear_operat} enjoys 
$$
u\otimes v:\nabla n\in C^{1,\beta}(\partial \Omega).
$$
Consequently, by putting \eqref{est_interp_double_1}, \eqref{est_interp_double_2}, \eqref{est_interp_double_3} and \eqref{est_interp_double_4} all together, we can apply \cref{t:bil_interp} with $X_1=C_\sigma^\beta(\Omega)$, $X_2=C_\sigma^{1,\beta}(\Omega)$, $Y_1=C_\sigma^\beta(\Omega)$ and $Y_2=C_\sigma^{2,\beta}(\Omega)$, getting
\begin{equation}
\label{est_interp_double_5}
\|T(u,u)\|_{\left( C^\beta(\Omega),C^{2,\beta}(\Omega)\right)_{\theta-\beta,\infty}}\leq C \|u\|^2_{\left( C^\beta_\sigma (\Omega),C^{1,\beta}_\sigma (\Omega)\right)_{\theta-\beta,\infty}}
\end{equation}
as soon as $\theta>\beta$. 
We refer to \cref{s:bilin_interp} for the precise definition of the interpolation spaces used above, and to the classical monographs~\cites{L,BL} for more general discussions about linear interpolation.

Now recall that the Besov space $B^\gamma_{p,q}$ (see for instance \cite{BL} for the precise definition) coincides with the usual H\"older space whenever $p=q=\infty$ and $\gamma$ is not an integer. Thus,  by \cite{BL}*{Theorem~6.4.5}, it holds 
$$ \left( C^\beta(\Omega),C^{2,\beta}(\Omega)\right)_{\frac12,\infty}=C^{1,\beta}(\Omega),
$$ 
which in particular also shows the hypothesis \eqref{half_interp_hyp} in \cref{t:bil_interp} is satisfied. Moreover, again by \cite{BL}*{Theorem~6.4.5}, we deduce that
\begin{equation}
\label{interp_C0-C2} 
\left( C^\beta(\Omega),C^{2,\beta}(\Omega)\right)_{\theta-\beta,\infty}=C^{2\theta-\beta}(\Omega).
\end{equation}
We underline that~\cite{BL}*{Theorem~6.4.5} is stated in the whole space $\R^d$, from which \eqref{interp_C0-C2} easily follows by the existence of a linear extension operator which is also continuous between H\"older spaces.  Similarly, we have 
\begin{equation}
\label{interp_C0-C1} 
\left( C^\beta_\sigma(\Omega),C_\sigma^{1,\beta}(\Omega)\right)_{\theta-\beta,\infty}=C^{\theta}_\sigma(\Omega).
\end{equation}
Finally, by plugging \eqref{interp_C0-C2} and \eqref{interp_C0-C1} into \eqref{est_interp_double_5}, we obtain 
$$
\|T(u,u)\|_{C^{2\theta-\beta}(\Omega)}\leq C \|u\|^2_{C^\theta(\Omega)},
$$
from which \eqref{est_p_almost_double} follows since $\beta<\eps$.

\section{Proof of \texorpdfstring{\cref{double_torus}}{double regularity with arbitrary divergence on the torus}}
\label{non-zero_div}

In this section, we prove \cref{double_torus}.
To keep the notation short, we set $g=\diver u$ and we let $f$ be the unique zero-average solution of the problem 
\[
    \Delta f= g\quad \text{in } \T^d.
\]
We hence set $v=\nabla f$. 
We can now  deal with the two cases $\theta<\frac12$ and $\theta>\frac12$ separately.

\subsection{The case \texorpdfstring{$\theta<\frac12$}{theta<1/2}}

Without loss of generality, we can assume that $q<+\infty$. Since $g\in L^q(\T^d)$, by the standard Calder\'on--Zygmund theory we have
\begin{equation}\label{v_CZ}
\|v\|_{W^{1,q}(\T^d)}\leq C \| g\|_{L^q(\T^d)}
\end{equation}
which, in combination with the usual Sobolev embedding, being $q\geq \frac{2d}{1-2\theta}>\frac{d}{1-\theta}$, gives
\begin{align}
\|v\|_{C^\theta(\T^d)}\leq C \|v\|_{W^{1,q}(\T^d)}\leq C \| g\|_{L^q(\T^d)}\label{v_Ctheta}.
\end{align}
Now, rewriting 
\[
    u\otimes u=(u-v)\otimes (u-v)+(u-v)\otimes v+v\otimes (u-v)+v\otimes v,
\]
we can decompose the pressure as 
\[
    p=p_1+p_2+p_3,
\]
where the $p_i$'s are the unique zero-average solutions of 
\begin{align*}
-\Delta p_1 &=\diver \diver (w\otimes w),\\
-\Delta p_2&=2 \diver ((w\cdot \nabla) v),\\
-\Delta p_3 &=\diver \diver (v\otimes v),
\end{align*}
where we set $w=u-v$. 
Notice that, by definition of~$v$, we have $\diver w=0$ and, thanks to~\eqref{v_Ctheta}, $w\in C^\theta(\T^d)$.  

We now estimate the $C^{2\theta}$ norm of each $p_i$, $i=1,2,3$, separately.  
First, by~\cite{CD18}*{Theorem~1.1}, we have 
\begin{equation}
\label{est_p1}
\|p_1\|_{C^{2\theta}(\T^d)}\leq C \|w\|^2_{C^\theta(\T^d)}\leq C \left(\|u\|^2_{C^\theta(\T^d)}+\|v\|^2_{C^\theta(\T^d)} \right)\leq C  \left(\|u\|^2_{C^\theta(\T^d)}+\|g\|^2_{L^q(\T^d)} \right)
\end{equation}
thanks to~\eqref{v_Ctheta}.
Moreover, again by standard Calder\'on--Zygmund estimates, we can infer
\begin{align*}
\|p_2\|_{W^{1,q}(\T^d)}\leq C\|(w\cdot \nabla) v\|_{L^q(\T^d)}\leq C\|w\|_{C^\theta(\T^d)}\|v\|_{W^{1,q}(\T^d)}\leq C  \left(\|u\|^2_{C^\theta(\T^d)}+\|g\|^2_{L^q(\T^d)} \right)
\end{align*}
which, together with the Sobolev embedding $W^{1,q}(\T^d)\subset C^{2\theta}(\T^d)$ valid for $q\geq \frac{d}{1-2\theta}$, gives
\begin{equation}
\label{est_p2}
\|p_2\|_{C^{2\theta}(\T^d)}\leq C\|p_2\|_{W^{1,q}(\T^d)}\leq C \left(\|u\|^2_{C^\theta(\T^d)}+\|g\|^2_{L^q(\T^d)} \right).
\end{equation}
Finally, in a similar way, we have 
\[
    \|p_3\|_{W^{1,\frac{q}{2}}(\T^d)}\leq C\|v\otimes v\|_{W^{1,\frac{q}{2}}(\T^d)}\leq C\|v\|_{W^{1,q}(\T^d)}^2 \leq C \| g\|_{L^{q}(\T^d)}^2
\]
which, again in combination with the Sobolev embedding $W^{1,\frac{q}{2}}(\T^d)\subset C^{2\theta}(\T^d)$ valid for $q\geq \frac{2d}{1-2\theta}$, leads to
\begin{equation}\label{est_p3}
\|p_3\|_{C^{2\theta}(\T^d)}\leq C \|p_3\|_{W^{1,\frac{q}{2}}(\T^d)}\leq C \| g\|_{L^{q}(\T^d)}^2.
\end{equation}
Combining~\eqref{est_p1}, \eqref{est_p2} and~\eqref{est_p3} all together, we complete the proof of \cref{double_torus}~\eqref{item:torus<1/2}.

\subsection{The case \texorpdfstring{$\theta>\frac12$}{theta>1/2}}

By standard Schauder estimates, we infer that
\begin{equation}
\label{v_schauder}
\|v\|_{C^{1,2\theta-1}(\T^d)}\leq C\|g\|_{C^{2\theta-1}(\T^d)}.
\end{equation}
We now split $p=p_1+p_2+p_3$ exactly as in the previous case and estimate the H\"older norm of each piece separately.  
First, by \cite{CD18}*{Theorem~1.1}, together with~\eqref{v_schauder}, we can bound
\[
    \|p_1\|_{C^{1,2\theta-1}(\T^d)}\leq C\left(\|u\|^2_{C^\theta(\T^d)}+\|v\|^2_{C^{\theta}(\T^d)} \right)\leq C\left(\|u\|^2_{C^\theta(\T^d)}+\|g\|^2_{C^{2\theta-1}(\T^d)} \right).
\]
Moreover, again by standard Schauder estimates, we have
\begin{align*}
\|p_2\|_{C^{1,2\theta-1}(\T^d)}&\leq C \| (w\cdot \nabla) v\|_{C^{2\theta-1}(\T^d)} \leq C \| w \|_{C^{2\theta-1}(\T^d)} \| v \|_{C^{1,2\theta-1}(\T^d)}\\
&\leq C\left(\|u\|^2_{C^\theta(\T^d)}+\|g\|^2_{C^{2\theta-1}(\T^d)}  \right),
\end{align*}
where the last inequality follows from the embedding $C^\theta(\T^d)\subset C^{2\theta-1}(\T^d)$. 
Finally, once again by Schauder estimates, we can write
\[
    \|p_3\|_{C^{1,2\theta-1}(\T^d)}\leq C\|v\|^2_{C^{1,2\theta-1}(\T^d)}\leq C\|g\|^2_{C^{2\theta-1}(\T^d)}.
\]
The validity of \cref{double_torus}~\eqref{item:torus>1/2} then follows by combining the above estimates and the proof is complete.

\appendix 

\section{Schauder regularity estimates}

\label{s:schauder}

\subsection{Elliptic regularity estimates}

For the reader's convenience, below we state some well-known elliptic regularity estimates that are used throughout the paper.

We start with the following local regularity estimates. 
For the proofs, we refer the reader to \cite{GT}*{Theorem~4.15 and Theorem~6.26}. 

\begin{theorem}[Local estimates]\label{thm.ellipticRegularity}
Let $\alpha\in(0,1)$, $f\in C^\alpha(B(0,2))$ and $g\in C^{1,\alpha}(\Gamma(0,2))$, where 
\begin{equation*}
\Gamma(0,2)=\partial B^+(0,2)\cap\{x_d=0\}, 
\qquad
B^+(0,2)=B(0,2)\cap\{x_d>0\}.
\end{equation*}

\begin{enumerate}[(i)]

\item\label{item:elliptic_ball}
If $\Delta u = f$ in $B(0,2)$, then
    \begin{equation}    
\label{eq.schauderInterior}
        \|u\|_{C^{2,\alpha}(B(0,1))} \leq C_d \left( \|u\|_{L^{\infty}(B(0,2))} + \|f\|_{C^{\alpha}(B(0,2))}\right),
    \end{equation}
where $C_d>0$ is a dimensional constant.

\item\label{item:elliptic_half-ball}
Let $A\in C^{1,\alpha}(B(0,2);\R^{d\times d})$ be a uniformly elliptic matrix with
\begin{equation*}
A\xi\cdot\xi\ge\nu|\xi|^2
\qquad
\xi\in\R^d,\ \nu>0,
\end{equation*}
and let $b\in C^{1,\alpha}(\Gamma(0,2);\R^d)$ be such that $b(x)\cdot\mathrm{e}_d \neq 0$ for all $x\in \Gamma(0,2)$.
If $u\in C^{2,\alpha}(B(0,2)\cup\Gamma(0,2))$ is a solution of
    \[
    \left\{ 
    \begin{array}{rcll}
        \diver  (A\nabla u)  & = & f &\text {in } B^+(0,2)\\[1mm]
        b\cdot\nabla u & = & g &\text{on } \Gamma(0,2), 
    \end{array}
    \right. 
    \]
then
    \begin{equation}
        \label{eq.schauderBoundary}
        [u]_{C^{2,\alpha}(B^+(0,1))}\leq C\left(\|u\|_{L^{\infty}(B^+(0,2))} +\|f\|_{C^{\alpha}(B^+(0,2))} + \|g\|_{C^{1,\alpha}(\Gamma(0,2))}\right),  
    \end{equation}
where the constant $C>0$ depends on $\nu>0$ and the $C^{1,\alpha}$ norms of~$A$ and~$b$ only.
\end{enumerate}
\end{theorem}

The following result can be seen as a `one-derivative-less' counterpart of the Schauder estimates for the Poisson problem with Neumann boundary condition established in~\cite{N14}.
We also refer the reader to~\cite{L13}*{Section~4} for more general results in this direction.

\begin{theorem}[Global estimates with Neumann boundary condition]
\label{nardi-1}
Let $\alpha\in(0,1)$ and let $\Omega\subset\R^d$ be a bounded simply connected open set of class $C^{1,\alpha}$.
Let $g\in C^\alpha(\partial\Omega)$, $F\in C^{\alpha}(\Omega;\R^d)$ be such that 
\begin{equation*}
-\int_{\partial\Omega} g\,d\sigma(x)=\int_{\partial \Omega} F\cdot n\,d\sigma (x).
\end{equation*} 
There exists a solution $u\in C^{1,\alpha}(\Omega)$ (unique up to an additive constant) of the problem
\begin{equation*}
   \left\{ 
    \begin{array}{rcll}
        -\Delta u &= & \diver F &\text {in}\ \Omega\\[1mm]
\partial_nu & = & g &\text{on}\ \partial\Omega. 
    \end{array}
    \right. 
\end{equation*} 
Moreover, every solution of this problem verifies the estimate
\begin{equation*}
\left\|\,u-\frac1{|\Omega|}\int_\Omega u\,dx\,\right\|_{C^{1,\alpha}(\Omega)}
\le
C(d,\alpha,\Omega)\left( \|g\|_{C^\alpha(\partial\Omega)}+ \|F\|_{C^\alpha(\Omega)}\right).
\end{equation*}
\end{theorem} 
\begin{proof}
From \cite{V22}*{Theorem 1.2 and Corollary 1.5} we get a solution $u\in C^{1,\alpha}(\Omega)$ such that 
\[
    \|u\|_{C^{1,\alpha}(\Omega)}\leq C\left(\|g\|_{C^\alpha(\partial\Omega)}+ \|F\|_{C^\alpha(\Omega)}+ \|u\|_{C^{0}(\Omega)} \right).
\]
Removing the $\|u\|_{C^{0}(\Omega)} $ from the right-hand side of the previous estimate is a standard contradiction argument, see for instance \cite{N14}*{First proof of Theorem 4.1}.
\end{proof}

We also recall some local Calder\'on--Zygmund type estimates for Neumann elliptic problems. 

\begin{theorem}[Local $L^p$ regularity estimates]\label{thm.LpLocalRegularity} 
Let $p\in(1,+\infty)$, $\rho \in (1,2]$ and $f \in L^p(B(0,2))$.

\begin{enumerate}[(i)]
    \item \label{item:interiorLp} If $-\Delta u = f$ in $B(0,2)$ then 
    \begin{equation}
        \label{eq.interiorCZ}
        \|u\|_{W^{2,p}(B(0,1))} \leq C(d,\rho)\left( \|u\|_{L^p(B(0,\rho))} + \|f\|_{L^p(B(0,\rho))}\right). 
    \end{equation}
    \item \label{item:boundaryLp} Let $A\in C^{1,\alpha}(B(0,2);\R^{d\times d})$ be a uniformly elliptic matrix with
\begin{equation*}
A\xi\cdot\xi\ge\nu|\xi|^2
\qquad
\xi\in\R^d,\ \nu>0,
\end{equation*}
and let $b\in C^{1,\alpha}(\Gamma(0,2);\R^d)$ be such that $b(x)\cdot\mathrm{e}_d \neq 0$ for all $x\in \Gamma(0,2)$. Let $g\in W^{1-\frac{1}{p},p}(\Gamma(0,2))$. If $u \in W^{2,p}(B(0,2))$ is a solution of 
\[
    \left\{ 
    \begin{array}{rcll}
        \diver  (A\nabla u)  & = & f &\text {in } B^+(0,2)\\[1mm]
        b\cdot\nabla u & = & g &\text{on } \Gamma(0,2), 
    \end{array}
    \right. 
\]
then 
\begin{equation}
    \label{eq.boundaryCZ}
    \|u\|_{W^{2,p}(B(0,1))} \leq C(d,\rho, \nu) \left(\|u\|_{L^p(B(0,\rho))}+\|f\|_{L^p(B(0,1))} + \|g\|_{W^{1-\frac{1}{p},p}(\Gamma(0,\rho))}\right).
\end{equation}
\end{enumerate}
\end{theorem}

\begin{proof} 
For the proof of~\eqref{item:interiorLp}, we refer to \cite{GT}*{Theorem~9.11}. To prove~\eqref{item:boundaryLp} in the case $g=0$, one can first use an extension procedure as in \cite{GT}*{Theorem~9.13}, where the same estimate is derived for Dirichlet boundary conditions. 
The general case can be reduced to the case $g=0$ by considering $u-G$ instead of $u$, where $G \in W^{2,p}(B^+(0,2))$ satisfies $\partial_n G = g$ and $\|G\|_{W^{2,p}(B^+(0,2))}\leq C(d,p) \|g\|_{W^{1-1/p,p}(\Gamma (0,2))}$. 
\end{proof}

We will also need the following global $L^p$ regularity estimate. 
For the proof, we refer the reader to~\cite{Am20}.

\begin{theorem}[Global $L^p$ regularity estimates]\label{thm.LpGlobalRegularity}
Let $\Omega\subset\R^d$ be a $C^{2,\alpha}$ bounded domain for some $\alpha\in(0,1)$. Let $p\in(1,+\infty)$ and let $f \in L^p(\Omega)$ be such that $\diver f\in L^1(\Omega)$.  
If $u$ is a weak solution of 
\begin{equation}
    \label{eq.NeumanGreen}
    \left\{ 
    \begin{array}{rcll}
        -\Delta u  & = & \diver f - \dfrac1{|\Omega|}\displaystyle\int_\Omega\diver f\,dx &\text {in } \Omega\\[3mm]
        \partial_n u & = & 0 &\text{on } \partial\Omega, 
    \end{array}
    \right. 
\end{equation}
then 
\begin{equation}
        \label{eq.globalCZ}
        \|u\|_{W^{1,p}(\Omega)} \leq C(p,\Omega) \left(\|f\|_{L^p} + 1\right). 
    \end{equation}
\end{theorem}

The following consequence of \cref{thm.LpGlobalRegularity} will be of particular importance in the proof of the pointwise estimates for the Neumann--Green function in \cref{greenfunctionestimates} below.
For the proof, we refer to~\cite{DM}*{Lemma~1}. 

\begin{corollary}[Global $L^{p,\infty}$ regularity estimates]
\label{coro.GlobalWeakLp} 
Let $\Omega\subset\R^d$ be a $C^{2,\alpha}$ domain for some $\alpha\in(0,1)$.
Let $p\in(1,+\infty)$ and $f \in L^{p,\infty}(\Omega)$.
If $u$ is a weak solution of 
\[
    \left\{ 
    \begin{array}{rcll}
        -\Delta u  & = & \diver f - \frac{1}{|\Omega|} \langle \diver f, 1\rangle &\text {in } \Omega\\[3mm]
        \partial_n u & = & 0 &\text{on } \partial\Omega, 
    \end{array}
    \right. 
\]
then
\begin{equation}
        \label{eq.globalWeakCZ}
        \|\nabla u \|_{L^{p,\infty}(\Omega)} \leq C(p,\Omega)\left(\|f\|_{L^{p,\infty}(\Omega)}+1\right).
    \end{equation}
\end{corollary}

\subsection{An interpolation inequality}

We close this section with the following simple interpolation inequality.
Although we need to apply \cref{interpolation.lemma} only to the ball or to (a smooth regularization of) the half-ball, we state it in the general case.

\begin{lemma}[Interpolation]\label{interpolation.lemma} 
Let $\Omega\subset\R^d$ be a bounded connected open set of class $C^{2,\alpha}$ for some $\alpha \in (0,1)$. 
There exists a constant $C>0$, depending on $\Omega$ only, with the following property.
If $f\in C^{2,\alpha}(\Omega)$ and $k\in\{0,1,2\}$, then
\begin{equation}
\label{eq.interpolation2}
    [f]_{C^k(\Omega)} \leq C \|f\|_{C^0(\Omega)}^{\frac{k}{2+\alpha}}\|f\|_{C^{2,\alpha}(\Omega)}^{1-\frac{k}{2+\alpha}}. 
\end{equation}
\end{lemma}

\section{Abstract bilinear interpolation}\label{s:bilin_interp}
Here we recall a particular case of the abstract bilinear interpolation theorem from \cite{CDF20}*{Theorem 3.5}. We start by defining the basic notions in order to state the result, while we refer the reader to \cites{BL,L} for a detailed overview about Interpolation Theory.
\\
\\
Let $(X,\|\cdot\|_X)$ and $(Y,\|\cdot\|_Y)$ be two real Banach spaces. The couple $(X,Y)$ is said to be an interpolation  couple if both $X$ and $Y$ are continuously embedded in a topological Hausdorff vector space.  Moreover, we recall the definition of the $K$-function, by introducing the following notation. Given $x\in X+Y$ we denote $\Omega(x)=\{  (a,b)\in X\times Y : a+b=x\}\subset X\times Y$, and for all $t>0$ we define
\begin{equation*}
K(t,x,X,Y)=\inf\Big\{\|a\|_X+t\|b\|_Y : (a,b)\in\Omega(x)\Big\}.
\end{equation*}
Consequently, for any $\theta\in(0,1)$,  we set
\[
(X,Y)_{\theta,\infty}=\left\{x\in X+Y \hbox{ s.t. } t\mapsto t^{-\theta}K(t,x)
\in L^\infty([0,\infty])\right\},
\]
which turns out to be a Banach space when endowed with the norm 
\begin{equation*}
\|x\|_{(X,Y)_{\theta,\infty}}=\left\|(\cdot)^{-\theta}K(\cdot,x)\right\|_{L^\infty([0,\infty])}.
\end{equation*}
\begin{theorem}[Bilinear interpolation from \cite{CDF20}*{Theorem 3.5}]\label{t:bil_interp}
Let $(X_1,X_2)$ and $(Y_1,Y_2)$ be two interpolation couples. Let $T$ be a bilinear operator satisfying 
\begin{align*}
\|T(a_1,a_2)\|_{Y_1}&\leq C\|a_1\|_{X_1}\|a_2\|_{X_1},\\\label{bil:est:2}
\|T(b_1,b_2)\|_{Y_2}&\leq C\|b_1\|_{X_2}\|b_2\|_{X_2},
\end{align*}
and
\begin{equation}\label{half_interp_hyp}
\|T(a,b)\|_{(Y_1,Y_2)_{\frac12,\infty}}+\|T(b,a)\|_{(Y_1,Y_2)_{\frac12,\infty}}\leq C\|a\|_{X_1}\|b\|_{X_2},
\end{equation}
for some constant $C>0$ independent of $a,a_1,a_2\in X_1$ and $b,b_1,b_2\in X_2$. Then, for any $\theta\in(0,1)$, there holds
\begin{equation*}
\|T(x,x)\|_{(Y_1,Y_2)_{\theta,\infty}}\leq C\|x\|^2_{(X_1,X_2)_{\theta,\infty}} 
\end{equation*}
for all $x\in(X_1,X_2)_{\theta,\infty}$.
\end{theorem}

\section{Pointwise estimates on the Green--Neumann function}
\label{greenfunctionestimates}
 
In this section, we establish some estimates on the \emph{Green--Neumann function} on a sufficiently regular domain $\Omega\subset\R^d$ for $d\ge2$ we need in the paper. 
Basically, these estimates assert that the behavior of \emph{Green--Neumann function} is comparable to that of the corresponding \emph{Newtonian potential}
\begin{equation}
\label{newtonian_potential}
\pot(x)
=
    \left\{
    \begin{array}{rcl}
\dfrac{1}{\omega_d(d-2)}\,\dfrac1{|x|^{d-2}} &\text{for } d\ge3
\\[6mm]
-\dfrac{1}{2\pi}\,\log|x| & \text{for } d=2
    \end{array}
    \right.
\end{equation} 
with $x\in\R^d$, $x\ne0$.

On a bounded connected open set $\Omega\subset\R^d$ of class $C^{1}$, the \emph{Green--Neumann function} $G=G(x,y)$  (see~\cite{DB10}*{Chapter~3}) is a solution of
\[
    \left\{
    \begin{array}{rcll}
         -\Delta G(x,\,\cdot\,) &=& \delta_x - \frac{1}{|\Omega|} &\text{in } \Omega  \\[2mm]
         \partial_n G(x,\, \cdot\,)&=&0 &\text{on } \partial\Omega, 
    \end{array}
    \right.
\]
whenever $x\in\Omega$, where as usual $\delta_x$ denotes the Dirac measure at~$x$.
Possibly replacing $G(x,y)$ with $G(x,y)-v(x)$, where 
\[
    v(x)=\frac{1}{|\Omega|}\int_{\Omega}G(x,y)\,dy,
\] 
it is not restrictive to assume that~$G$ is a symmetric function, see~\cite{DB10}*{Chapter~3, Lemma~7.1} for the proof of this statement.
We refer the reader to~\cite{DB10}*{Chapter~3, Section~7} for more details on the main properties of  the Green--Neumann function.

The main estimates on the Green--Neumann function $G$ we use in this paper are gathered in the following result. 

\begin{theorem}[Pointwise estimates of the Green function]\label{l:greenestimate}
Let $\Omega\subset\R^d$, $d\ge2$, be a bounded connected open set of class $C^{2,\alpha}$ for some $\alpha\in(0,1)$.
There exists a constant $C>0$, depending on $\Omega$ only, with the following properties.

\begin{enumerate}[(i)]

\item\label{item:neu.der1} 
If $\beta\in\N_0^d$ is such that $|\beta| \leq 2$, then
\begin{equation}
\label{eq.pointwiseOrder12}
        |\partial_y^{\beta}G(x,y)| \leq \frac{C(\Omega, \beta)}{|x-y|^{d-2+|\beta|}} 
    \end{equation}
for all $x,y\in\Omega$. 

\item\label{item:neu.der2}
If $x_1,x_2 \in\Omega$, $\bar x := \frac{x_1+x_2}{2}$ and $h:=|x_1-x_2|$, then
    \begin{equation}
        \label{eq.pointwiseDifference}
        |\partial_{y_iy_j}G(x_1,y)-\partial_{y_iy_j}G(x_2,y)|\leq \frac{Ch}{|\bar x - y|^{d+1}} 
\end{equation}
for all $y\in \Omega\setminus B(\bar x,h)$ and $i,j=1,\dots,d$.
\end{enumerate}
\end{theorem}

The proofs of~\eqref{item:neu.der1} and~\eqref{item:neu.der2} are very similar and follow the simple argument outlined in~\cite{DM}, which we readapt to the Neumann boundary case. 
To keep this article short, we prove \cref{l:greenestimate} for $d\ge 3$ only.
The proof of \cref{l:greenestimate} for  $d=2$ follows the same strategy with the usual minor adaptations depending on the different expression of the Newtonian potential~\eqref{newtonian_potential}.

\begin{proof}[Proof of \cref{l:greenestimate}]
For $r_0>0$, we let 
\[
    \Omega_{r_0}=\left\{x\in\overline \Omega :  \operatorname{dist}(x,\partial\Omega) \leq r_0\right\}.
\]
By compactness, we can cover $\overline\Omega$ with balls of radius~$r_0$. This yields a finite covering of the set~$\Omega_{r_0}$,
\[
    \left\{
    B(c_k,r_0) : c_k\in\overline\Omega,\ k=1,\dots,K
    \right\}
\]
depending on the chosen~$r_0$.
Possibly choosing a smaller~$r_0$ if needed, one can ensure that, for each $k=1,\dots,K$, there exists a $C^{2,\alpha}$-diffeomorphism 
\[
\Phi_k \colon U_{k} \to B(c_k,16r_0) \cap \Omega,
\qquad
U_k\subset\R^d\ \text{open},
\] 
such that 
\[
    \Phi_k^{-1}(c_k)=0,
    \qquad
    B^+(0,8r_0) \subset U_k \subset B^+(0,32r_0).
\] 
Without loss of generality, we can further assume that 
\[
    \frac{3}{4}\leq \frac{|\Phi_k(x)-\Phi_k(y)|}{|x-y|} \leq \frac{5}{4}
\]
for all $x,y\in U_k$, $x\ne y$.

\medskip

Now let $x, y \in\Omega$ and set $r=|x-y|$. 
In the following, we assume that $r \leq r_0/2$. 
This condition will be removed in the last part of the proof.

\medskip

\noindent
\textit{Proof of~\eqref{item:neu.der1} for $r\le r_0/2$}.
We distinguish two cases. 

\medskip

\noindent\textbf{Case 1 $B(y,r/34)\subset \Omega$.}
Since the function 
\[
    \psi_x =\psi(x,\cdot) = G(x,\cdot)-\pot(x-\cdot)
\]
solves
\[
    \Delta \psi_x = \frac{1}{|\Omega|} \quad \text{in } B(y,r/34),
\]
the function 
\[
    \tilde{\psi}(z):=\psi_x(y+rz/68),
\quad \text{for}\ z\in B(0,2),
\]
solves 
\[
    \Delta\tilde{\psi}_x = C(\Omega)\, r^2\quad \text{in}\ B(0,2),
\]
where $C(\Omega)>0$ is a constant depending on $\Omega$ only (that may change from line to line in what follows). 
We now distinguish two subcases, $\beta=0$ and $|\beta|\ge1$.

\medskip

\noindent
\textbf{Subcase 1.1 $\beta=0$.} Let $p_0>\frac{d}{2}$, such that $W^{2,p_0} \hookrightarrow L^{\infty}$. Using the $W^{2,p}$ regularity estimate~\eqref{eq.interiorCZ} applied to $\psi_x$ with exponent $p_0$ and $\rho \in (1,2)$ sufficiently close to $1$; followed by the Sobolev embedding $W^{2,p_1}(B(0,\rho)) \hookrightarrow L^{p_0}(B(0,\rho))$ where $\frac{1}{p_1} \in \left(\frac{1}{p_0}+\frac{2}{d},1\right)$ to obtain 
\begin{align*}
    \|\tilde{\psi}_x\|_{W^{2,p_0}((B(0,1))} &\leq C(d,p_0,\rho) \left(\|\tilde{\psi}_x\|_{L^{p_0}(B(0,\rho))} + 1 \right) \\ 
    &\leq C(d,p_0,\rho) \left(\|\tilde{\psi}_x\|_{W^{2,p_1}(B(0,\rho))} + 1 \right).
\end{align*}
We are now in a position to iterate this application of \cref{thm.LpLocalRegularity}, constructing a decreasing sequence $p_n >1$. One only needs to iterate theses estimates $n=n(d)$ times so that $p_n\in (1,\frac{d}{d-2})$ and therefore 
\[
    \|\tilde{\psi}_x\|_{W^{2,p_0}((B(0,1))} \leq C(d,p_0,\rho)\left(\|\tilde{\psi}_x\|_{L^{p_n}(B(0,\rho^n))} + 1 \right).
\]
We choose $\rho = \rho(d)$ such that $\rho^n <2$. Now observe that, because $p_n<\frac{d}{d-2}$, there is $\varepsilon >0$ such that $\dot{W}^{1,\frac{d}{d-1}-\varepsilon}(B(0,2)) \hookrightarrow L^{p_n}(B(0,2))$. Using this and the inequality
\[
    \|\nabla \tilde{\psi}_x\|_{L^{\frac{d}{d-1} -\varepsilon}(B(0,2))} \leq C(d) \|\nabla \tilde{\psi}_x\|_{L^{\frac{d}{d-1}, \infty}(B(0,2))},
\]
we arrive at 
\begin{equation*}
    \|\tilde{\psi}_x\|_{W^{2,p_0}((B(0,1))} \leq C(d,p_0,\rho) \left(\|\nabla \tilde{\psi}_x\|_{L^{\frac{d}{d-1},\infty}(B(0,2))}+1\right). 
\end{equation*}
Using the Sobolev embedding $W^{2,p_0}(B(0,1)) \hookrightarrow L^{\infty}(B(0,1))$ and undoing the scaling, we finally have 
\begin{equation}
    \label{eq.firstBoundOrder0}
    \|\psi_x\|_{L^{\infty}(B(y,r/68))} \leq C(d) \left( \frac{1}{r^{d-2}}\|\nabla \psi_x\|_{L^{\frac{d}{d-1},\infty}(\Omega)}+1\right).
\end{equation}
Since $\|\nabla \phi(x - \cdot)\|_{L^{\frac{d}{d-1},\infty}(\Omega)}<\infty$, it follows that 
\[
    \|\nabla \psi_x\|_{L^{\frac{d}{d-1},\infty}(\Omega)} \leq C + \|\nabla G(x, \cdot)\|_{L^{\frac{d}{d-1},\infty}(\Omega)}.
\]
Observe that  
\[
    \left\{
    \begin{array}{rcll}
         -\Delta G(x,\,\cdot\,) &=& \diver f - \frac{1}{|\Omega|} &\text{in } \Omega  \\[2mm]
         \partial_n G(x,\, \cdot\,)&=&0 &\text{on } \partial\Omega , 
    \end{array}
    \right.
\]
with $f(y) = \frac{(x-y)}{d\omega_d|x-y|^{d}} \in L^{\frac{d}{d-1},\infty}$, so that \cref{coro.GlobalWeakLp} yields 
\[
    \|\nabla G(x,\cdot)\|_{L^{\frac{d}{d-1},\infty}(\Omega)} \leq C(d) \left(\|f\|_{L^{\frac{d}{d-1},\infty}(\Omega)} + 1\right),
\]
hence
\begin{equation}
    \label{eq.weakBoundPsi}
    \|\nabla \psi_x\|_{L^{\frac{d}{d-1},\infty}(\Omega)} \leq C(d) < +\infty. 
\end{equation}
Combining the above inequality with~\eqref{eq.firstBoundOrder0}, we obtain 
\[
    \|\psi_x\|_{L^{\infty}(B(y,r/68))} \leq \frac{C(d)}{r^{d-2}}. 
\]
which ends the proof.

\medskip

\noindent
\textbf{Subcase 1.2 $|\beta|\ge 1$.} By the elliptic regularity estimate~\eqref{eq.schauderInterior} in \cref{thm.ellipticRegularity}, it follows that 
\begin{equation}
    \label{eq.firstBound}
    \|\tilde{\psi}_x\|_{C^{2,\alpha}(B(0,1))} \leq C_d\left(\|\tilde{\psi}_x\|_{L^{\infty}(B(0,2))} +C(\Omega)\,r^2\right).
\end{equation}
Note that, by the zero-order estimates on $G$, that is, by~\eqref{eq.pointwiseOrder12} for $\beta=0$, one has
\[
    |\psi_x(z)| \leq |G(x,z)| + |\pot(x-z)| \leq \frac{C(\Omega)}{|x-z|^{d-2}} \leq \frac{C(\Omega)}{r^{d-2}}
\]
for all $z \in B(y,r/34)$, so that 
\begin{equation}
    \label{eq.order0psi}
    \|\tilde{\psi}_x\|_{L^{\infty}(B(0,2))} \leq \frac{C(\Omega)}{r^{d-2}}.
\end{equation}
The above inequality, combined with~\eqref{eq.firstBound}, gives 
\begin{equation}
    \label{eq.order2psi}
    \|\tilde{\psi}_x\|_{C^{2,\alpha}(B(0,1))}\leq \frac{C(\Omega)}{r^{d-2}}.
\end{equation}
Interpolating~\eqref{eq.order0psi} and~\eqref{eq.order2psi} using~\eqref{eq.interpolation2}, one obtains $[\tilde{\psi}_x]_{C^{|\beta|}(B(0,1))} \leq \frac{C(\Omega)}{r^{d-2}}$, that is
\[
    [\psi_x]_{C^{|\beta|}(B(y,r/68))} \leq \frac{C(\Omega)}{r^{d-2+|\beta|}},
\]
which is enough to conclude. 

\medskip 

\noindent
\textbf{Case 2 $B(y,r/34)\cap\partial\Omega \neq \varnothing$.} Let us pick any point $a\in B(y,r/34)\cap\partial\Omega$. 
We observe that, since $r\leq r_0/2$, there holds that $y\in B(c_k,r_0)$ for some $1 \leq k \leq K$ that we fix. 
Because $|y-a|\leq r/34$, we have 
\[
    B(a,r/17) \cap \Omega \subset B(c_k,r_0+r/17) \subset B(c_k,2r_0).
\]
Note that $Y:=\Phi_k^{-1}(a) \in B^+(0,4r_0)$, so that $B^+(Y,8r/17) \subset B^+(0,8r_0) \subset U_k$. We also note that $B(y,r/34)\cap \Omega \subset B(a,r/17) \cap \Omega \subset \Phi_k(B(Y,2r/17))$. In \cref{fig.1} we have drawn some balls appearing in the analysis. 
\begin{figure}
    \centering
    \includegraphics[width=\textwidth]{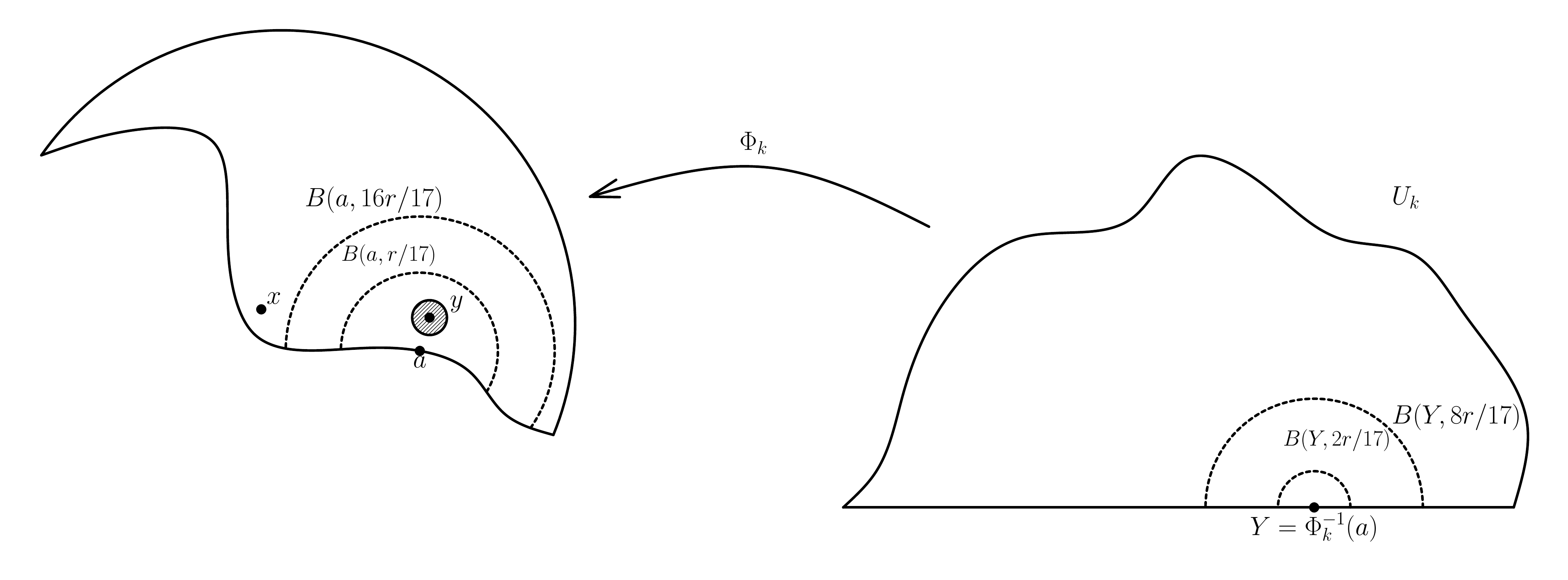}
    \caption{Some balls appearing in the analysis.  The goal is to estimate the function $\psi_x$ on the shaded ball.}
    \label{fig.1}
\end{figure}

We are now ready to start the proof.
Since
\[
\left\{
\begin{array}{rclll}
    \Delta \psi_x & = & \frac{1}{|\Omega|} =: C &\text{ in } &B(y,r/2)\cap\Omega \\[2mm]
    \partial_{n}\psi_x  & = & -\partial_n\phi (x-\cdot) &\text{ on } &B(y,r/2) \cap \partial\Omega, 
\end{array}
\right.
\]
the function $\theta_x:= \theta \circ \Phi_k$ satisfies 
\[
\left\{
\begin{array}{rclll}
    \diver  (A(z)\nabla \theta_x(z)) & = & C &\text{ in } &B^+(Y,8r/17) \\[2mm]
    b(z)\cdot \nabla \theta_x(z)  & = & -\partial_n\phi (x-\Phi_k(\cdot)) &\text{ on } &\Gamma(Y,8r/17), 
\end{array}
\right.
\]
where the matrix $A$ and the vector $b$ depend only on the $C^{2,\alpha}$ diffeomorphism $\Phi_k$, in particular $A$ and $b$ does not depend on $r$. Note that by our premiliminary remarks we can ensure that $B_+(Y,8r/17) \subset U_k$. Let us set $\tilde{A}(z)=A(Y+4rz/17)$ and $\tilde{b}(z)=b(Y+4rz/17)$ which are also of class $C^{1,\alpha}$ uniformly in $r$ (and $y$)  since $r \leq r_0/2$, which is fixed. Letting $\tilde{\theta}_x= \theta_x(Y+4r \cdot /17)$ and $\tilde{g} := - r \partial_n \phi (x-\Phi_k(Y+4r\cdot/17))$, we can write 
\[
\left\{
\begin{array}{rclll}
    \diver  (\tilde{A}\nabla \tilde{\theta}_x) & = & Cr^2 &\text{ in } &B^+(0,2) \\[2mm]
    \tilde{b}\cdot\nabla \tilde{\theta}_x& = & \tilde{g} &\text{ on } &\Gamma(0,2).  
\end{array}
\right.
\]
We now again distinguish two subcases, $\beta=0$ and $|\beta|\ge1$.

\medskip

\noindent
\textbf{Subcase 2.1 $\beta =0$.} 
We start with an application of the $W^{2,p}$ regularity estimate~\eqref{eq.boundaryCZ} applied to $\tilde{\theta}_x$, $\rho \in (1,2)$ and $p_0>\frac{d}{2}$, which, in turn, is such that we have the Sobolev embedding $W^{2,p_0}(B(0,\rho)) \hookrightarrow L^{\infty}(B(0,\rho))$. This writes 
\begin{equation}
    \label{eq.firstApplicationBoundaryReg}
    \|\tilde{\theta}_x\|_{W^{2,p_0}(B^+(0,1))} \leq C(d,\rho) \left(\|\tilde{\theta}_x\|_{L^{p_0}(B^+(0,\rho))} + \|\tilde{g}\|_{W^{1-\frac{1}{p_0},p_0} (B^+(0,\rho))}+1\right).  
\end{equation}
Observe that 
\[
\|\tilde{g}\|_{W^{1-\frac{1}{p_0},p_0}(B^+(0,\rho))} \leq \|\tilde{g}\|_{W^{1,p_0}(B^+(0,\rho))} \leq \|\tilde{g}\|_{L^{p_0}(B^+(0,\rho))} + \|\nabla \tilde{g}\|_{L^{p_0}(B^+(0,\rho))}.
\] We have 
\begin{align*}
    \|\nabla \tilde{g}\|_{L^{p_0}(B^+(0,\rho))} &\leq r^2 \|\nabla ^2 \phi (x-\Phi_k(Y+4r \cdot /17))\|_{L^{p_0}(B^+(0,2))} \\
    &\leq C(\Omega) r^{2-\frac{d}{p_0}} \|\nabla^2 \phi (x - \cdot)\|_{L^{p_0}(B(y,\frac{33r}{34}))},
\end{align*}
where we used that
\begin{equation}
    \label{eq.inclusions}
    \Phi_k(B^+(Y,8r/17)) \subset B(a,16r/17) \subset B(y,33r/34) \cap \Omega 
\end{equation}
and the properties of $\Phi_k$ as well as a scaling change of variable. 
Now since $|x-z| \ge \frac{r}{34}$ for any $z \in B(y,33r/34)$ and that $|\nabla ^2 \phi (x -z)| \leq  \frac{C(d)}{|x-z|^{d}}$, we finally obtain that 
\[
    \|\nabla\tilde{g}\|_{L^{p_0}(B^+(0,\rho))} \leq \frac{C(\Omega)}{r^{d-2}},
\]
and similarly the same estimate holds for
$\|\tilde{g}\|_{L^{p_0}(B^+(0,\rho))}$, so that ultimately we get 
\[
    \|\tilde{g}\|_{W^{1-\frac{1}{p_0},p_0}(B^+(0,\rho))} \leq \frac{C(\Omega)}{r^{d-2}}. 
\]
Importantly our computations do not depend on the value of $p_0$. For this reason we are in a position to iterate~\eqref{eq.firstApplicationBoundaryReg} sufficiently many time just as we did in \textbf{Subcase~1.1.} and obtain
\[
    \|\tilde{\theta}_x\|_{W^{2,p_0}(B^+(0,1))} \leq C(\Omega) \left(\|\tilde{\theta}_x\|_{L^{p_n}(B^+(0,2)))} + \frac{1}{r^{d-2}} + 1 \right),
\]
where $n=n(d)$ is chosen so that $p_n \in (1,\frac{d}{d-2})$ (note that we also need to chose $\rho$ such that $\rho ^n \leq 2$). Using the Sobolev embedding $\dot{W}^{1,\frac{d}{d-1}- \varepsilon}(B^+(0,2))  \hookrightarrow L^{p_n}(B^+(0,2))$ and the embedding $L^{\frac{d}{d-1},\infty}(B^+(0,2)) \hookrightarrow L^{\frac{d}{d-1}-\varepsilon}(B^+(0,2))$, one finally obtains that
\[
    \|\tilde{\theta}_x\|_{W^{2,p_0}(B^+(0,1))} \leq C(\Omega) \left(\|\nabla \tilde{\theta}_x\|_{L^{\frac{d}{d-1},\infty}(B^+(0,2)))} + \frac{1}{r^{d-2}} + 1 \right).
\]
Going back to the original variables and undoing the scaling we get 
\[
    \|\nabla \tilde{\theta}_x\|_{L^{\frac{d}{d-1},\infty}(B^+(0,2)))} \leq \frac{C(\Omega)}{r^{d-2}} \|\nabla \psi_x\|_{L^{\frac{d}{d-1},\infty}(\Omega)}, 
\]
and thanks to the $L^{\frac{d}{d-1},\infty}$ bound on $\psi_x$ given by~\eqref{eq.weakBoundPsi} we arrive at 
\begin{equation}
    \label{eq.almostFinalBoundary}
    \|\tilde{\theta}_x\|_{W^{2,p_0}(B^+(0,1))} \leq \frac{C(\Omega)}{r^{d-2}}. 
\end{equation}
To conclude, observe that 
\begin{align*}
    \|\psi_x\|_{L^{\infty}(B(y,r/34)\cap\Omega)} &\leq \|\psi_x\|_{L^{\infty}(B(a,r/17)\cap\Omega)}\\
    &\leq C(\Omega) \|\theta_x\|_{L^{\infty}(B^+(Y,2r/17))} \\
    &\leq C(\Omega) \|\tilde{\theta}_x\|_{L^{\infty}(B^+(0,1/2))},
\end{align*}
which combined with the Sobolev embedding $W^{2,p_0} \hookrightarrow L^{\infty}$ and~\eqref{eq.almostFinalBoundary} provides us with the desired estimate
\[
    \|\psi_x\|_{L^{\infty}(B(y,r/34)\cap\Omega)}\leq \frac{C(\Omega)}{r^{d-2}}. 
\]

\medskip

\noindent
\textbf{Subcase 2.2 $|\beta| \geq 1$.} 
An application of~\eqref{eq.schauderBoundary} to $\tilde{\theta}_x$ yields
\[
    \|\tilde{\theta}_x\|_{C^{2,\alpha}(B^+(0,1)}\leq C(\Omega) \left(\|\tilde{\theta_x}\|_{L^{\infty}(B^+(0,2))} + \|Cr^2\|_{C^{\alpha}(B^+(0,2))} + \|g\|_{C^{1,\alpha}(\Gamma(0,2))}\right). 
\]
Note that $\|Cr^2\|_{C^{\alpha}(B^+(0,2))}\leq Cr^2$. Then, using that 
\[
    \Phi(B^+(Y,8r/17)) \subset B(a,16r/17) \subset B(y,33r/34) \cap \Omega 
\] 
and also that $\Phi_k$ is a $C^{2,\alpha}$ diffeomorphism, it follows that  
\begin{equation}
    \label{eq.order0r}
    \|\tilde{\theta}_x\|_{L^{\infty}(B^+(0,2))} = \|\theta_x\|_{L^{\infty}(B^+(Y,8r/17))} \leq C(\Omega) \|\psi_x\|_{L^{\infty}(B(y,33r/34))} \leq \frac{C(\Omega)}{r^{d-2}}, 
\end{equation}
where in the last step we have used the zero-order estimate. 
Similarly, to estimate $\tilde g$,
by rescaling and using that $\Phi_k$ is a $C^{2,\alpha}$ diffeomorphism, one obtains
\begin{align*}
    [\tilde g]_{C^{1,\alpha}(\Gamma^+(0,2))} & = Cr^{2+\alpha}[\partial_n \phi (x - \Phi_k(\cdot))]_{C^{1,\alpha}(\Gamma(Y,8r/17))}\\
    &\leq  Cr^{2+\alpha} \|\partial_n \phi (x - \cdot) \|_{C^{1,\alpha}(B(y,33r/34)\cap \partial \Omega)} \\
    &\leq C r^{2+\alpha}\|\phi (x - \cdot) \|_{C^{2,\alpha}(B(y,33r/34)\cap \partial \Omega)}.    
\end{align*}
Now, we observe that 
\begin{align*}
   \|\phi (x - \cdot) \|_{C^{2,\alpha}(B(y,33r/34)\cap\partial\Omega)} &\leq \|\phi (x - \cdot) \|_{L^{\infty}(B(y,33r/34)\cap\partial\Omega)} \\
   &+ [\phi (x - \cdot)]_{C^{2,\alpha}(B(y,33r/34)\cap\partial\Omega)} \\
   & \leq \frac{C(\Omega)}{r^{d-2}} + [\phi (x - \cdot)]_{C^{2,\alpha}(B(y,33r/34)\cap\partial\Omega)}.
\end{align*}
Moreover, we have
\[
    |\partial_j\partial_m \phi(x-z_1) - \partial_j\partial_m\phi(x-z_2)| \leq \frac{C(d)}{r^{d+\alpha}} |z_1-z_2|^{\alpha},
\]
for all $z_1,z_2 \in B(y,33r/34)$ and all $j, m \in \{1, \dots, d\}$, which is obtained by explicit computations on $\phi(x-z)=\frac{C(d)}{|x-z|^{d-2}}$. Gathering all these estimates, it follows that
\begin{equation}
    \label{eq.order2r}
        \|\tilde{\theta}_x\|_{C^{2,\alpha}(B^+(0,1))} \leq \frac{C}{r^{d-2}}. 
\end{equation}
Interpolating~\eqref{eq.order2r} with~\eqref{eq.order0r} exploiting~\eqref{eq.interpolation2} on a smooth domain $U^+(0,1)$ such that $B^+(0,1/2) \subset U^+(0,1)\subset B^+(0,1)$, we obtain
\[
    [\tilde{\theta}_x]_{C^{|\beta|}(B^+(0,1/2))} \leq [\tilde{\theta}_x]_{C^{|\beta|}(U^+(0,1))} \leq \frac{C(\Omega)}{r^{d-2}}. 
\]
Also, note that $\|\tilde{\theta}_x\|_{L^{\infty}(B^+(0,1/2))} \leq \frac{C(\Omega)}{r^{d-2}}$, so that we can finally write 
\begin{align*}
   [\psi_x]_{C^{|\beta|}(B(y,r/34)\cap\Omega)}&\leq [\psi_x]_{C^{|\beta|}(B(a,r/17)\cap\Omega)} \leq C(\Omega) \|\theta_x\|_{C^{|\beta|}(B^+(Y,2r/17))} \\
   &\leq \frac{C(\Omega)}{r^{|\beta|}}[\tilde{\theta}_x]_{C^{|\beta|}(B^+(0,1/2))} + C(\Omega)\|\tilde{\theta}_x\|_{L^{\infty}(B^+(0,1/2))}\\
   &\leq \frac{C(\Omega)}{r^{d-2+|\beta|}},
\end{align*}
which is sufficient to conclude the proof. 

\medskip

\noindent 
\textit{Proof of~\eqref{item:neu.der2} for $r\le r_0/2$}.
Let us denote $|x_1-x_2|=h$ and $r :=|\bar x - y|$. Let $g_{x_1,x_2} := \psi(x_1,\cdot) - \psi(x_2, \cdot)$. Since~\eqref{eq.pointwiseDifference} is true for $\phi$ in place of $G$ (by direct computations), the proof of~\eqref{eq.pointwiseDifference} directly follows from the inequality 
\begin{equation}
    \label{eq.requiredDiff}
    [g_{x_1,x_2}]_{C^{2}(B(y,r/68)\cap\Omega)} \leq \frac{h}{r^{d+1}}, 
\end{equation}
for some global constant $c>0$. As in the proof of~\eqref{item:neu.der1}, we distinguish between the case where $y$ is far from the boundary or close to it. 

We can assume that $h \leq \frac{r}{34}$.
Indeed, if $h>\frac{r}{34}$, then we can just apply the triangle inequality and~\eqref{item:neu.der1} to obtain 
\begin{align*}
    |\partial_i \partial_j g(x_1,y) - \partial_i\partial_j g(x_2,y)| &\leq |\partial_i \partial_j g(x_1,y)| + |\partial_i\partial_j g(x_2,y)| \\
   & \leq C(\Omega)\left(\frac{1}{|x_1 - y|^d} + \frac{1}{|x_2 - y|^d} \right) \\
   &\leq \frac{C(\Omega)}{r^d}\leq C(\Omega) \frac{h}{r^{d+1}}.
\end{align*}
As in the proof of~\eqref{item:neu.der1}, we  distinguish two cases.

\medskip

\noindent
\textbf{Case 1 $B(y,r/34) \subset \Omega$.} 
The function $g_{x_1,x_2}$ satisfies the equation 
\[  
    \Delta_z g_{x_1,x_2} = 0 \quad \text{in}\ B(y,r/34),
\]
therefore the function $\tilde{g}(z):=g_{x_1,x_2}(y+rz/68)$ satisfies 
\[
    \Delta \tilde{g}=0 \quad\text{in}\ B(0,2),    
\]
so that, using the interior Schauder estimates~\eqref{eq.schauderInterior} on $\tilde{g}$, one has 
\[
    \|\tilde{g}\|_{C^{2,\alpha}(B(0,1))} \leq C(\Omega)\|\tilde{g}\|_{L^{\infty}(B(0,2))}.    
\]
Assuming the estimate 
\begin{equation}
    \label{eq.psi1}
    \|g_{x_1,x_2}\|_{L^{\infty}(B(y,r/34))}\leq C(\Omega)\frac{h}{r^{d-1}},
\end{equation}
it follows that 
\[
    \|\tilde{g}\|_{L^{\infty}(B(0,2))} \leq C(\Omega)\frac{h}{r^{d-1}},  
\]
so that 
\begin{equation}
    \label{eq.psi11}
    \|\tilde{g}\|_{C^{2,\alpha}(B(0,1))} \leq \frac{C(\Omega)h}{r^{d-1}}.    
\end{equation}
Interpolating~\eqref{eq.psi1} and~\eqref{eq.psi11} using~\eqref{eq.interpolation2} and then undoing the scaling, one obtains the estimate~\eqref{eq.requiredDiff}. Therefore the proof of~\eqref{eq.pointwiseDifference} boils down to the proof of~\eqref{eq.psi1}.

\medskip 

\noindent
\textbf{Case 2 $B(y,r/34) \cap \partial\Omega \neq \varnothing$.} Let us fix $1\leq k\leq K$ such that $y \in B(c_k,r_0)$. The argument follows the same line as in~\eqref{item:neu.der1}, so that we should only sketch the argument below. The equation satisfied by $g_{x_1,x_2}$ on $B(y,r/2)\cap \Omega$ is 
\[
\left\{
\begin{array}{rclll}
    \Delta_z g_{x_1,x_2} & = & 0 &\text{ in } &B(y,r/2)\cap\Omega\\[2mm]
    \partial_{n}g_{x_1,x_2} & = & \partial_n \phi(x_2-\cdot)-\partial_n\phi(x_1 -\cdot) &\text{ in } &B(y,r/2)\cap\partial\Omega.  
\end{array}
\right.
\]
Let $a\in B(y,r/34)\cap\partial \Omega$ and set $Y=\Phi_k^{-1}(a)$. The function $\theta:=g_{x_1,x_2}(\Phi_k(Y + 4r \cdot /17))$ satisfies the equation 
\[
\left\{
\begin{array}{rclll}
    \diver  (A\nabla \theta) & = & 0 &\text{ in } & B^+(0,2)\\[2mm]
    b\cdot\nabla \theta & = & G &\text{ in } &\Gamma(0,2),
\end{array}
\right.
\]
where  $A$ is some uniformly elliptic matrix and $b$ a vector both with $C^{1,\alpha}$ norms independent of $r$, and 
\[
G=r\left(\partial_n \phi(x_2-\Phi_k(Y+4r\cdot /17))-\partial_n\phi(x_1 -\Phi_k(Y+4r\cdot /17))\right).
\]
In particular, an application of~\eqref{eq.schauderBoundary} yields 
\begin{equation*}
    \|\theta\|_{C^{2,\alpha}(B^+(0,1))} \leq C(\Omega)\left(\|\theta\|_{L^{\infty}(B^+(0,2))} + \|G\|_{C^{1,\alpha}(\Gamma(0,2))}\right).
\end{equation*}
Observe that, by rescaling and using that $\Phi_k$ is a $C^{2,\alpha}$ diffeomorphism, one has 
\begin{align*}
    [G]_{C^{1,\alpha}(\Gamma(0,2))} 
    &= r^{2+\alpha}[\partial_n \phi(x_2-\Phi_k(\cdot))-\partial_n\phi(x_1 -\Phi_k(\cdot))]_{C^{1,\alpha}(B^+(Y,8r/17))}\\
  &  \leq C(\Omega) r^{2+\alpha}\|\partial_n \phi(x_2-\cdot)-\partial_n\phi(x_1 -r\cdot)\|_{C^{1,\alpha}(B(a,16r/17)\cap\partial\Omega)} \\
    &\leq C(\Omega)r^{2+\alpha}\|\partial_n \phi(x_2-\cdot)-\partial_n\phi(x_1 -r\cdot)\|_{C^{1,\alpha}(B(y,33r/34)\cap\partial\Omega)}. 
\end{align*}
By explicit computations on $\phi$ and by recalling that $h \leq \frac{r}{34}$, one has 
\[
   [\partial_n \phi(x_2-\cdot)-\partial_n\phi(x_1 -\cdot)]_{C^{1,\alpha}(B(y,33r/34)\cap\partial\Omega)} \leq C(\Omega)\frac{h}{r^{d+1+\alpha}},  
\]
so that 
\[
    [G]_{C^{1,\alpha}(\Gamma(0,2))} \leq C(\Omega) \frac{h}{r^{d-1}}.
\]
Similarly, one can estimate  
\[
    \|G\|_{L^{\infty}(\Gamma(0,2))} \leq C(\Omega)\frac{h}{r^{d-1}},
\]
so that we finally have 
\[
    \|G\|_{C^{1,\alpha}(\Gamma(0,2))} \leq C(\Omega)\frac{h}{r^{d-1}}. 
\]
Similarly, we have 
\[
    \|\theta\|_{L^{\infty}(B^+(0,1))} \leq C(\Omega) \|g_{x_1,x_2}\|_{L^{\infty}(B(y,33r/34)\cap\Omega)}.  
\]
Assuming the validity of the bound 
\begin{equation}
    \label{eq.psi2}
    \|g_{x_1,x_2}\|_{L^{\infty}}(B(y,33r/34)\cap\Omega) \leq C(\Omega)\frac{h}{r^{d-1}},  
\end{equation}
the above inequalities lead to the estimate 
\begin{equation}
    \label{eq.psi4}
    \|\theta\|_{C^{2,\alpha}(B^+(0,1))}\leq C(\Omega) \frac{h}{r^{d-1}}. 
\end{equation}
Interpolating the rescaled version of~\eqref{eq.psi2} with~\eqref{eq.psi4} using \cref{interpolation.lemma}, and following the same arguments as in the proof of~\eqref{item:neu.der1}, one obtains~\eqref{eq.requiredDiff}, provided that~\eqref{eq.psi2} holds. 

\medskip
\noindent\textbf{Proof of \eqref{eq.psi1} and \eqref{eq.psi2}.} 
We need to prove that the function $\psi(x,\cdot)=\psi_x$ satisfies
\begin{equation}
    \label{eq.psi3}
    |\psi(x_1,z) - \psi(x_2,z)| \leq C(\Omega) \frac{h}{r^{d-1}} \text{ for all } z \in \Omega \cap B(y,33r/34). 
\end{equation}
First, let us observe that, by the triangle inequality, one has
\[
    |\psi(x_1,z) - \psi(x_2,z)| \leq |\psi(x_1,z)| + |\psi(x_1,z)| \leq C(\Omega) \frac{1}{r^{d-2}} \leq C(\Omega) \frac{h}{r^{d-1}},
\]
as soon as $r \leq C(\Omega) h$. Therefore, we can assume that $r \geq 3 h$. Note that, since $\psi(x_1,z) - \psi(x_2,z)$ is a difference, we can assume that $x_1$ and $x_2$ are either both in the same ball of radius $r_0$ of the chosen covering of $\overline \Omega$, or in two different but neighboring balls. 
Hence, we can assume that $x_1, x_2 \in B(c_k,r_0)$ for some $k$. 
We first explain how to construct a path $\gamma \colon [0,1] \to \Omega\cap B(\bar{x},h)$ joining $x_1$ to $x_2$. Recall that $\Phi_k\colon U_k \to B(c_k,16r_0)\cap \Omega$ is a diffeomorphism such that $\Phi_k^{-1}(B(\bar{x},h))$ is contained in the convex set $B^+(0,8r_0) \subset U_k$. Therefore, we can set $X_j=\Phi_k^{-1}(x_j)$ for $j=1, 2$ and define the path 
\[
\gamma (t) = \Phi_k((1-t)X_1+tX_2),
\]
$t\in[0,1]$, which has length less than $2|x_1-x_2|$ since $\Phi_k$ is a diffeomorphism, see \cref{fig.2}. Also, again because $\Phi_k$ is a diffeomorphism, we have that $|x'-y| \geq r/3$ for all $x'\in \gamma([0,1])$, being $r \geq 3h$ according to our initial assumption.    
\begin{figure}
    \centering
    \includegraphics[width=\textwidth]{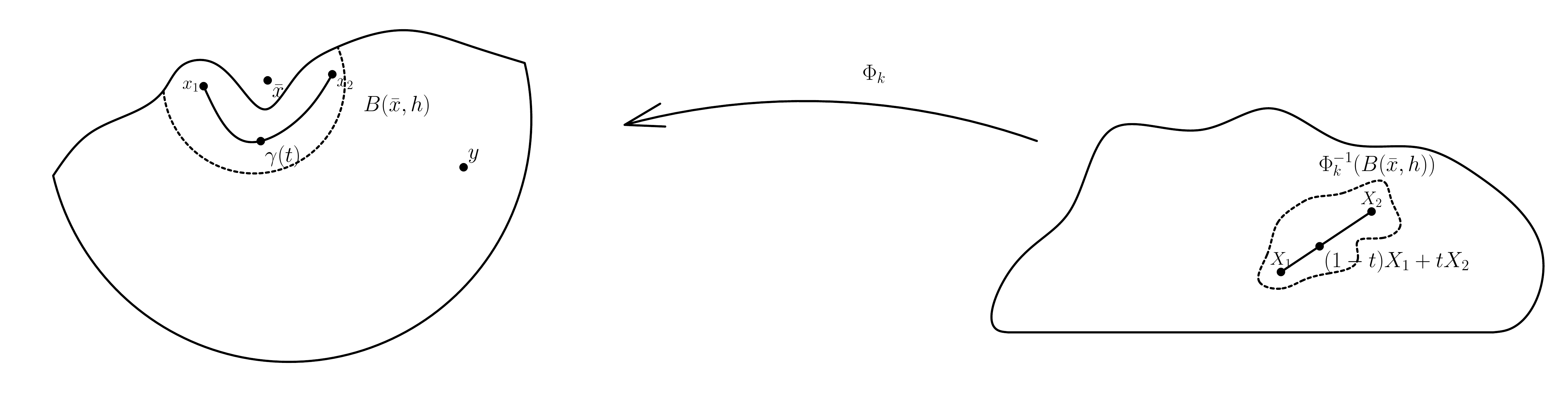}
    \caption{Construction of the path $\gamma(t)$.}
    \label{fig.2}
\end{figure}
We can now start the proof by observing that, by the Mean Value Theorem, 
\[
    \psi(x_1,z)-\psi(x_2,z)= \int_{0}^{1}\nabla_x \psi(\gamma (t),z) \cdot\gamma'(t)\,dt, 
\]
and thus
\[
    |\psi(x_1,z)-\psi(x_2,z)| \leq C(\Omega)|x_1-x_2| \|\nabla_x G(\cdot, z)\|_{L^{\infty}(B(\bar x, h))}.
\]
Because of the symmetry $G(x,z)=G(z,x)$ we see that $\nabla _x G(x,z)=\nabla_y G(z,x)$, so that by~\eqref{item:neu.der1} we have the uniform bound
\[
    |\nabla_x G(x',z)| \leq \frac{C(\Omega)}{|x'-z|^{d-1}} \leq \frac{C(\Omega)}{r^{d-1}} 
\]
for all $z\in B(y,33r/34)$ and $x'\in\gamma([0,1])$,
which finally gives~\eqref{eq.psi3}.

\medskip

To conclude the proof, we are thus left to prove~\eqref{item:neu.der1} and~\eqref{item:neu.der2} for $r \geq r_0/2$.
In this case, one can write all the estimates in the proofs of~\eqref{item:neu.der1} and of~\eqref{item:neu.der2} with $r_0$ instead of $r$ and end up with bounds in terms of negative powers of $r_0$, which are then bounded by negative powers of $r$ up to enlarging the constants.
\end{proof}

\begin{remark}
One may wonder if the estimates on $\psi_x=\psi(x,\cdot)$ given in the above proof are optimal. 
In fact, the argument presented above should be interpreted as a scaling obstruction for a larger growth than $|x-\,\cdot\,|^{-d+2-|\beta|}$ for $\psi_x$. Since this bound has the same order of the one that can be computed for the Newtonian potential, this is enough to conclude. 
\end{remark}


\begin{bibdiv}
\begin{biblist}



\bib{Am20}{article}{
   author={Amrouche, C.},
   author={Conca, C.},
   author={Ghosh, A.},
   author={Ghosh, T.},
   title={Uniform $W^{1,p}$ estimates for an elliptic operator with Robin boundary condition in a $\mathcal{C}^1$ domain},
   journal={Calc. Var. Partial Differential Equations},
   volume={59},
   date={2020},
   number={2},
   pages={Paper No. 71, 25},
}

\bib{BT18}{article}{
   author={Bardos, Claude},
   author={Titi, Edriss S.},
   title={Onsager's conjecture for the incompressible Euler equations in
   bounded domains},
   journal={Arch. Ration. Mech. Anal.},
   volume={228},
   date={2018},
   number={1},
   pages={197--207},
}

\bib{BT21}{article}{
   author={Bardos, Claude},
   author={Titi, Edriss S.},
   title={$C^{0,\alpha}$  boundary regularity for the pressure in weak solutions of the $2d$ Euler equations},
   journal={Phil. Trans. R. Soc. A.}, 
   volume={380},
   date={2022},
}


\bib{BTW19}{article}{
   author={Bardos, Claude},
   author={Titi, Edriss S.},
   author={Wiedemann, Emil},
   title={Onsager's conjecture with physical boundaries and an application
   to the vanishing viscosity limit},
   journal={Comm. Math. Phys.},
   volume={370},
   date={2019},
   number={1},
   pages={291--310},
}


\bib{BL}{book}{
   author={Bergh, J\"{o}ran},
   author={L\"{o}fstr\"{o}m, J\"{o}rgen},
   title={Interpolation spaces. An introduction},
   note={Grundlehren der Mathematischen Wissenschaften, No. 223},
   publisher={Springer-Verlag, Berlin-New York},
   date={1976},
   pages={x+207},
}

\bib{BP19}{article}{
   author={Berselli, Luigi C.},
   author={Longo, Placido},
   title={Classical solutions for the system ${\text{curl}\, v = g}$, with
   vanishing Dirichlet boundary conditions},
   journal={Discrete Contin. Dyn. Syst. Ser. S},
   volume={12},
   date={2019},
   number={2},
   pages={215--229},
}


\bib{BP20}{article}{
   author={Berselli, Luigi C.},
   author={Longo, Placido},
   title={Classical solutions of the divergence equation with Dini
   continuous data},
   journal={J. Math. Fluid Mech.},
   volume={22},
   date={2020},
   number={2},
   pages={Paper No. 26, 20},
}

\bib{BDLSV2019}{article}{
   author={Buckmaster, Tristan},
   author={de Lellis, Camillo},
   author={Sz\'{e}kelyhidi, L\'{a}szl\'{o}, Jr.},
   author={Vicol, Vlad},
   title={Onsager's conjecture for admissible weak solutions},
   journal={Comm. Pure Appl. Math.},
   volume={72},
   date={2019},
   number={2},
   pages={229--274},
}

\bib{CCFS08}{article}{
   author={Cheskidov, A.},
   author={Constantin, P.},
   author={Friedlander, S.},
   author={Shvydkoy, R.},
   title={Energy conservation and Onsager's conjecture for the Euler
   equations},
   journal={Nonlinearity},
   volume={21},
   date={2008},
   number={6},
   pages={1233--1252},
}
	
\bib{CD18}{article}{
   author={Colombo, Maria},
   author={De Rosa, Luigi},
   title={Regularity in time of H\"{o}lder solutions of Euler and
   hypodissipative Navier-Stokes equations},
   journal={SIAM J. Math. Anal.},
   volume={52},
   date={2020},
   number={1},
   pages={221--238},
}

\bib{CDF20}{article}{
   author={Colombo, Maria},
   author={De Rosa, Luigi},
   author={Forcella, Luigi},
   title={Regularity results for rough solutions of the incompressible Euler
   equations via interpolation methods},
   journal={Nonlinearity},
   volume={33},
   date={2020},
   number={9},
   pages={4818--4836},
}

\bib{C2014}{article}{
   author={Constantin, P.},
   title={Local formulas for hydrodynamic pressure and their applications},
   language={Russian, with Russian summary},
   journal={Uspekhi Mat. Nauk},
   volume={69},
   date={2014},
   number={3(417)},
   pages={3--26},
   translation={
      journal={Russian Math. Surveys},
      volume={69},
      date={2014},
      number={3},
      pages={395--418},
   },
}
		
\bib{CET94}{article}{
   author={Constantin, Peter},
   author={E, Weinan},
   author={Titi, Edriss S.},
   title={Onsager's conjecture on the energy conservation for solutions of
   Euler's equation},
   journal={Comm. Math. Phys.},
   volume={165},
   date={1994},
   number={1},
   pages={207--209},
}

\bib{DS2013}{article}{
   author={De Lellis, Camillo},
   author={Sz\'{e}kelyhidi, L\'{a}szl\'{o}, Jr.},
   title={Dissipative continuous Euler flows},
   journal={Invent. Math.},
   volume={193},
   date={2013},
   number={2},
   pages={377--407},
}
		
\bib{DS2014}{article}{
   author={De Lellis,Camillo},
   author={Sz\'{e}kelyhidi, L\'{a}szl\'{o}, Jr.},
   title={Dissipative Euler flows and Onsager's conjecture},
   journal={J. Eur. Math. Soc. (JEMS)},
   volume={16},
   date={2014},
   number={7},
   pages={1467--1505},
}

\bib{DH21}{article}{
   author={De Rosa, Luigi},
   author={Haffter, Silja},
   title={Dimension of the singular set of wild H\"older solutions of the incompressible Euler equations},
   status={preprint},
   date={2021},
   eprint={https://arxiv.org/abs/2102.06085},
}

\bib{DT19}{article}{
   author={De Rosa, Luigi},
   author={Tione, Riccardo},
   title={Sharp energy regularity and typicality results for H\"{o}lder
   solutions of incompressible Euler equations},
   journal={Anal. PDE},
   volume={15},
   date={2022},
   number={2},
   pages={405--428},
}


\bib{DB10}{book}{
   author={DiBenedetto, Emmanuele},
   title={Partial differential equations},
   publisher={Birkh\"{a}user Boston, Inc., Boston, MA},
   date={1995},
}

\bib{DM}{article}{
   author={Dolzmann, G.},
   author={M\"{u}ller, S.},
   title={Estimates for Green's matrices of elliptic systems by $L^p$
   theory},
   journal={Manuscripta Math.},
   volume={88},
   date={1995},
   number={2},
   pages={261--273},
}

\bib{Ey94}{article}{
   author={Eyink, Gregory L.},
   title={Energy dissipation without viscosity in ideal hydrodynamics. I.
   Fourier analysis and local energy transfer},
   journal={Phys. D},
   volume={78},
   date={1994},
   number={3-4},
   pages={222--240},
}

\bib{F}{book}{
   author={Frisch, Uriel},
   title={Turbulence},
   note={The legacy of A. N. Kolmogorov},
   publisher={Cambridge University Press, Cambridge},
   date={1995},
}

\bib{GT}{book}{
   author={Gilbarg, David},
   author={Trudinger, Neil S.},
   title={Elliptic partial differential equations of second order},
   series={Classics in Mathematics},
   note={Reprint of the 1998 edition},
   publisher={Springer-Verlag, Berlin},
   date={2001},
}

\bib{Is2013}{article}{
   author={Isett, Philip},
   title={Regularity in time along the coarse scale flow for the incompressible Euler equations},
   status={preprint},
   eprint={https://arxiv.org/abs/1307.0565},
   date={2013},
}

\bib{Is2018}{article}{
   author={Isett, Philip},
   title={A proof of Onsager's conjecture},
   journal={Ann. of Math. (2)},
   volume={188},
   date={2018},
   number={3},
   pages={871--963},
}

\bib{KMPT00}{article}{
   author={Kato, Tosio},
   author={Mitrea, Marius},
   author={Ponce, Gustavo},
   author={Taylor, Michael},
   title={Extension and representation of divergence-free vector fields on
   bounded domains},
   journal={Math. Res. Lett.},
   volume={7},
   date={2000},
   number={5-6},
   pages={643--650},
}

\bib{K41}{article}{
   author={Kolmogoroff, A.},
   title={The local structure of turbulence in incompressible viscous fluid
   for very large Reynold's numbers},
   journal={C. R. (Doklady) Acad. Sci. URSS (N.S.)},
   volume={30},
   date={1941},
}

\bib{L13}{book}{
   author={Lieberman, Gary M.},
   title={Oblique derivative problems for elliptic equations},
   publisher={World Scientific Publishing Co. Pte. Ltd., Hackensack, NJ},
   date={2013},
}


\bib{L}{book}{
   author={Lunardi, Alessandra},
   title={Interpolation theory},
   series={Appunti. Scuola Normale Superiore di Pisa (Nuova Serie). [Lecture
   Notes. Scuola Normale Superiore di Pisa (New Series)]},
   edition={2},
   publisher={Edizioni della Normale, Pisa},
   date={2009},
   pages={xiv+191},
   isbn={978-88-7642-342-0},
   isbn={88-7642-342-0},
}

\bib{N14}{article}{
   author={Nardi, Giacomo},
   title={Schauder estimate for solutions of Poisson's equation with Neumann
   boundary condition},
   journal={Enseign. Math.},
   volume={60},
   date={2014},
   number={3-4},
   pages={421--435},
}


\bib{Ons49}{article}{
   author={Onsager, L.},
   title={Statistical hydrodynamics},
   journal={Nuovo Cimento (9)},
   volume={6},
   date={1949},
   number={Supplemento, 2 (Convegno Internazionale di Meccanica Statistica)},
   pages={279--287},
}
		
\bib{RRS18}{article}{
   author={Robinson, James C.},
   author={Rodrigo, Jos\'e L.},
   author={Skipper, Jack W. D.},
   title={Energy conservation for the Euler equations on $\mathbb{T}^2\times\mathbb{R}_+$ for weak solutions defined without reference to the pressure},
   journal={Asymptot. Anal.},
   volume={110},
   date={2018},
   number={3-4},
   pages={185--202},
}

\bib{RRS2018}{article}{
   author={Robinson, James C.},
   author={Rodrigo, Jos\'{e} L.},
   author={Skipper, Jack W. D.},
   title={Energy conservation in the 3D Euler equation on $\mathbb
   T^2\times\mathbb R_+$},
   conference={
      title={Partial differential equations in fluid mechanics},
   },
   book={
      series={London Math. Soc. Lecture Note Ser.},
      volume={452},
      publisher={Cambridge Univ. Press, Cambridge},
   },
   date={2018},
   pages={224--251},
}

\bib{S16}{book}{
   author={Salsa, Sandro},
   title={Partial differential equations in action},
   series={Unitext},
   volume={99},
   edition={3},
   note={From modelling to theory;
   La Matematica per il 3+2},
   publisher={Springer, [Cham]},
   date={2016},
}

\bib{SILV}{article}{
   author={Silvestre, Luis},
title={A non obvious estimate for the pressure},
 note={Unpublished note},
 eprint={http://math.uchicago.edu/~luis/preprints/pressureestimate.pdf},
year={2011},
}


\bib{V22}{article}{
   author={Vita, Stefano},
   title={Boundary regularity estimates in H\"{o}lder spaces with variable
   exponent},
   journal={Calc. Var. Partial Differential Equations},
   volume={61},
   date={2022},
   number={5},
   pages={Paper No. 166, 31},
}

\end{biblist}
\end{bibdiv}
\end{document}